\documentclass[11pt, reqno]{amsart}
\usepackage{amssymb}
\usepackage{eucal}
\usepackage{amsmath}
\usepackage{amscd}
\usepackage[dvips]{color}
\usepackage{multicol}
\usepackage[all]{xy}           %xypic macro for latex2.09
\usepackage{graphicx}
\usepackage{color}
\usepackage{colordvi}
\usepackage{xspace}
\usepackage{tikz}

\usepackage[active]{srcltx} %SRC Specials for DVI Searching

\begin{document}
\newcommand {\emptycomment}[1]{} %to remove paragraphs

\baselineskip=14pt
\newcommand{\nc}{\newcommand}
\newcommand{\delete}[1]{}
\nc{\mfootnote}[1]{\footnote{#1}} % Use this to show footnotes
\nc{\todo}[1]{\tred{To do:} #1}

%\delete{
\nc{\mlabel}[1]{\label{#1}}  % Use this to suppress names
\nc{\mcite}[1]{\cite{#1}}  % Use this to suppress names
\nc{\mref}[1]{\ref{#1}}  % Use this to suppress names
\nc{\mbibitem}[1]{\bibitem{#1}} % Use this to show number
%}

\delete{
\nc{\mlabel}[1]{\label{#1}  % Use the next two lines to show names
{\hfill \hspace{1cm}{\bf{{\ }\hfill(#1)}}}}
\nc{\mcite}[1]{\cite{#1}{{\bf{{\ }(#1)}}}}  % Use this lines to show names
\nc{\mref}[1]{\ref{#1}{{\bf{{\ }(#1)}}}}  % Use this lines to show names
\nc{\mbibitem}[1]{\bibitem[\bf #1]{#1}} % Use this to show name
}

%%%%%%%%%%%%%%%%%%%%%%%% Statements
\newtheorem{thm}{Theorem}[section]
\newtheorem{lem}[thm]{Lemma}
\newtheorem{cor}[thm]{Corollary}
\newtheorem{pro}[thm]{Proposition}
\newtheorem{ex}[thm]{Example}
\newtheorem{rmk}[thm]{Remark}
\newtheorem{defi}[thm]{Definition}
\newtheorem{pdef}[thm]{Proposition-Definition}
\newtheorem{condition}[thm]{Condition}

\renewcommand{\labelenumi}{{\rm(\alph{enumi})}}
\renewcommand{\theenumi}{\alph{enumi}}

\nc{\tred}[1]{\textcolor{red}{#1}}
\nc{\tblue}[1]{\textcolor{blue}{#1}}
\nc{\tgreen}[1]{\textcolor{green}{#1}}
\nc{\tpurple}[1]{\textcolor{purple}{#1}}
\nc{\btred}[1]{\textcolor{red}{\bf #1}}
\nc{\btblue}[1]{\textcolor{blue}{\bf #1}}
\nc{\btgreen}[1]{\textcolor{green}{\bf #1}}
\nc{\btpurple}[1]{\textcolor{purple}{\bf #1}}

\nc{\cm}[1]{\textcolor{red}{Chengming:#1}}
\nc{\yy}[1]{\textcolor{blue}{Yanyong: #1}}
%\nc{\lit}[2]{\textcolor{blue}{#1}{ \textcolor{purple}{(#2)}}}
\nc{\lit}[2]{\textcolor{blue}{#1}{}} %use this line instead of the previous one to show only the new changes
\nc{\yh}[1]{\textcolor{green}{Yunhe: #1}}

%%%%%%%%%%%%%% Matrix symbols.

\nc{\twovec}[2]{\left(\begin{array}{c} #1 \\ #2\end{array} \right )}
\nc{\threevec}[3]{\left(\begin{array}{c} #1 \\ #2 \\ #3 \end{array}\right )}
\nc{\twomatrix}[4]{\left(\begin{array}{cc} #1 & #2\\ #3 & #4 \end{array} \right)}
\nc{\threematrix}[9]{{\left(\begin{matrix} #1 & #2 & #3\\ #4 & #5 & #6 \\ #7 & #8 & #9 \end{matrix} \right)}}
\nc{\twodet}[4]{\left|\begin{array}{cc} #1 & #2\\ #3 & #4 \end{array} \right|}

\nc{\rk}{\mathrm{r}}
\newcommand{\g}{\mathfrak g}
\newcommand{\h}{\mathfrak h}
\newcommand{\pf}{\noindent{$Proof$.}\ }
\newcommand{\frkg}{\mathfrak g}
\newcommand{\frkh}{\mathfrak h}
\newcommand{\Id}{\rm{Id}}
\newcommand{\gl}{\mathfrak {gl}}
\newcommand{\ad}{\mathrm{ad}}
\newcommand{\add}{\frka\frkd}
\newcommand{\frka}{\mathfrak a}
\newcommand{\frkb}{\mathfrak b}
\newcommand{\frkc}{\mathfrak c}
\newcommand{\frkd}{\mathfrak d}
\newcommand {\comment}[1]{{\marginpar{*}\scriptsize\textbf{Comments:} #1}}
%%%%%%%%%%%%%%%%%%%%%%% symbols

\nc{\gensp}{V} % space of generators
\nc{\relsp}{\Lambda} %space of relations
\nc{\leafsp}{X}    %decoration space for leaves
\nc{\treesp}{\overline{\calt}} % space of labeled trees

\nc{\vin}{{\mathrm Vin}}    %decoration set of indices
\nc{\lin}{{\mathrm Lin}}    %decoration set of leaves

\nc{\gop}{{\,\omega\,}}     % generic binary operation
\nc{\gopb}{{\,\nu\,}}
\nc{\svec}[2]{{\tiny\left(\begin{matrix}#1\\
#2\end{matrix}\right)\,}}  % column vector
\nc{\ssvec}[2]{{\tiny\left(\begin{matrix}#1\\
#2\end{matrix}\right)\,}} % subscript column vector

\nc{\typeI}{local cocycle $3$-Lie bialgebra\xspace}
\nc{\typeIs}{local cocycle $3$-Lie bialgebras\xspace}
\nc{\typeII}{double construction $3$-Lie bialgebra\xspace}
\nc{\typeIIs}{double construction $3$-Lie bialgebras\xspace}

\nc{\bia}{{$\mathcal{P}$-bimodule ${\bf k}$-algebra}\xspace}
\nc{\bias}{{$\mathcal{P}$-bimodule ${\bf k}$-algebras}\xspace}

\nc{\rmi}{{\mathrm{I}}}
\nc{\rmii}{{\mathrm{II}}}
\nc{\rmiii}{{\mathrm{III}}}
\nc{\pr}{{\mathrm{pr}}}
\newcommand{\huaA}{\mathcal{A}}

\nc{\pll}{\beta}
\nc{\plc}{\epsilon}

\nc{\ass}{{\mathit{Ass}}}
\nc{\lie}{{\mathit{Lie}}}
\nc{\comm}{{\mathit{Comm}}}
\nc{\dend}{{\mathit{Dend}}}
\nc{\zinb}{{\mathit{Zinb}}}
\nc{\tdend}{{\mathit{TDend}}}
\nc{\prelie}{{\mathit{preLie}}}
\nc{\postlie}{{\mathit{PostLie}}}
\nc{\quado}{{\mathit{Quad}}}
\nc{\octo}{{\mathit{Octo}}}
\nc{\ldend}{{\mathit{ldend}}}
\nc{\lquad}{{\mathit{LQuad}}}

 \nc{\adec}{\check{;}} \nc{\aop}{\alpha}
\nc{\dftimes}{\widetilde{\otimes}} \nc{\dfl}{\succ} \nc{\dfr}{\prec}
\nc{\dfc}{\circ} \nc{\dfb}{\bullet} \nc{\dft}{\star}
\nc{\dfcf}{{\mathbf k}} \nc{\apr}{\ast} \nc{\spr}{\cdot}
\nc{\twopr}{\circ} \nc{\tspr}{\star} \nc{\sempr}{\ast}
\nc{\disp}[1]{\displaystyle{#1}}
\nc{\bin}[2]{ (_{\stackrel{\scs{#1}}{\scs{#2}}})}  %binomial coeff
\nc{\binc}[2]{ \left (\!\! \begin{array}{c} \scs{#1}\\
    \scs{#2} \end{array}\!\! \right )}  %binomial coeff
\nc{\bincc}[2]{  \left ( {\scs{#1} \atop
    \vspace{-.5cm}\scs{#2}} \right )}  %binomial coeff
\nc{\sarray}[2]{\begin{array}{c}#1 \vspace{.1cm}\\ \hline
    \vspace{-.35cm} \\ #2 \end{array}}
\nc{\bs}{\bar{S}} \nc{\dcup}{\stackrel{\bullet}{\cup}}
\nc{\dbigcup}{\stackrel{\bullet}{\bigcup}} \nc{\etree}{\big |}
\nc{\la}{\longrightarrow} \nc{\fe}{\'{e}} \nc{\rar}{\rightarrow}
\nc{\dar}{\downarrow} \nc{\dap}[1]{\downarrow
\rlap{$\scriptstyle{#1}$}} \nc{\uap}[1]{\uparrow
\rlap{$\scriptstyle{#1}$}} \nc{\defeq}{\stackrel{\rm def}{=}}
\nc{\dis}[1]{\displaystyle{#1}} \nc{\dotcup}{\,
\displaystyle{\bigcup^\bullet}\ } \nc{\sdotcup}{\tiny{
\displaystyle{\bigcup^\bullet}\ }} \nc{\hcm}{\ \hat{,}\ }
\nc{\hcirc}{\hat{\circ}} \nc{\hts}{\hat{\shpr}}
\nc{\lts}{\stackrel{\leftarrow}{\shpr}}
\nc{\rts}{\stackrel{\rightarrow}{\shpr}} \nc{\lleft}{[}
\nc{\lright}{]} \nc{\uni}[1]{\tilde{#1}} \nc{\wor}[1]{\check{#1}}
\nc{\free}[1]{\bar{#1}} \nc{\den}[1]{\check{#1}} \nc{\lrpa}{\wr}
\nc{\curlyl}{\left \{ \begin{array}{c} {} \\ {} \end{array}
    \right .  \!\!\!\!\!\!\!}
\nc{\curlyr}{ \!\!\!\!\!\!\!
    \left . \begin{array}{c} {} \\ {} \end{array}
    \right \} }
\nc{\leaf}{\ell}       % number of leafs
\nc{\longmid}{\left | \begin{array}{c} {} \\ {} \end{array}
    \right . \!\!\!\!\!\!\!}
\nc{\ot}{\otimes} \nc{\sot}{{\scriptstyle{\ot}}}
\nc{\otm}{\overline{\ot}}
\nc{\ora}[1]{\stackrel{#1}{\rar}}
\nc{\ola}[1]{\stackrel{#1}{\la}}%${\Bbb Z}$
\nc{\pltree}{\calt^\pl}
\nc{\epltree}{\calt^{\pl,\NC}}
\nc{\rbpltree}{\calt^r}
\nc{\scs}[1]{\scriptstyle{#1}} \nc{\mrm}[1]{{\rm #1}}
\nc{\dirlim}{\displaystyle{\lim_{\longrightarrow}}\,}
\nc{\invlim}{\displaystyle{\lim_{\longleftarrow}}\,}
\nc{\mvp}{\vspace{0.5cm}} \nc{\svp}{\vspace{2cm}}
\nc{\vp}{\vspace{8cm}} \nc{\proofbegin}{\noindent{\bf Proof: }}
%\nc{\proofbegin}{\begin{proof}} % AMS command
\nc{\proofend}{$\blacksquare$ \vspace{0.5cm}}
%\nc{\proofend}{\end{proof}} %AMS command
\nc{\freerbpl}{{F^{\mathrm RBPL}}}
\nc{\sha}{{\mbox{\cyr X}}}  %used to be \cyr
\nc{\ncsha}{{\mbox{\cyr X}^{\mathrm NC}}} \nc{\ncshao}{{\mbox{\cyr
X}^{\mathrm NC,\,0}}}
\nc{\shpr}{\diamond}    %Shuffle product
\nc{\shprm}{\overline{\diamond}}    %Shuffle product
\nc{\shpro}{\diamond^0}    %Shuffle product
\nc{\shprr}{\diamond^r}     %product on controlled trees
\nc{\shpra}{\overline{\diamond}^r}
\nc{\shpru}{\check{\diamond}} \nc{\catpr}{\diamond_l}
\nc{\rcatpr}{\diamond_r} \nc{\lapr}{\diamond_a}
\nc{\sqcupm}{\ot}
\nc{\lepr}{\diamond_e} \nc{\vep}{\varepsilon} \nc{\labs}{\mid\!}
\nc{\rabs}{\!\mid} \nc{\hsha}{\widehat{\sha}}
\nc{\lsha}{\stackrel{\leftarrow}{\sha}}
\nc{\rsha}{\stackrel{\rightarrow}{\sha}} \nc{\lc}{\lfloor}
\nc{\rc}{\rfloor}
\nc{\tpr}{\sqcup}
\nc{\nctpr}{\vee}
\nc{\plpr}{\star}
\nc{\rbplpr}{\bar{\plpr}}
\nc{\sqmon}[1]{\langle #1\rangle}
\nc{\forest}{\calf}
\nc{\altx}{\Lambda_X} \nc{\vecT}{\vec{T}} \nc{\onetree}{\bullet}
\nc{\Ao}{\check{A}}
\nc{\seta}{\underline{\Ao}}
\nc{\deltaa}{\overline{\delta}}
\nc{\trho}{\tilde{\rho}}

\nc{\rpr}{\circ}
%\nc{\apr}{\cdot}
\nc{\dpr}{{\tiny\diamond}}
\nc{\rprpm}{{\rpr}}

%%%%%%%%%%%%%%%%%%%%% roman fonts, in alphabetic order
\nc{\mmbox}[1]{\mbox{\ #1\ }} \nc{\ann}{\mrm{ann}}
\nc{\Aut}{\mrm{Aut}} \nc{\can}{\mrm{can}}
\nc{\twoalg}{{two-sided algebra}\xspace}
\nc{\colim}{\mrm{colim}}
\nc{\Cont}{\mrm{Cont}} \nc{\rchar}{\mrm{char}}
\nc{\cok}{\mrm{coker}} \nc{\dtf}{{R-{\rm tf}}} \nc{\dtor}{{R-{\rm
tor}}}
\renewcommand{\det}{\mrm{det}}
\nc{\depth}{{\mrm d}}
\nc{\Div}{{\mrm Div}} \nc{\End}{\mrm{End}} \nc{\Ext}{\mrm{Ext}}
\nc{\Fil}{\mrm{Fil}} \nc{\Frob}{\mrm{Frob}} \nc{\Gal}{\mrm{Gal}}
\nc{\GL}{\mrm{GL}} \nc{\Hom}{\mrm{Hom}} \nc{\hsr}{\mrm{H}}
\nc{\hpol}{\mrm{HP}} \nc{\id}{\mrm{id}} \nc{\im}{\mrm{im}}
\nc{\incl}{\mrm{incl}} \nc{\length}{\mrm{length}}
\nc{\LR}{\mrm{LR}} \nc{\mchar}{\rm char} \nc{\NC}{\mrm{NC}}
\nc{\mpart}{\mrm{part}} \nc{\pl}{\mrm{PL}}
\nc{\ql}{{\QQ_\ell}} \nc{\qp}{{\QQ_p}}
\nc{\rank}{\mrm{rank}} \nc{\rba}{\rm{RBA }} \nc{\rbas}{\rm{RBAs }}
\nc{\rbpl}{\mrm{RBPL}}
\nc{\rbw}{\rm{RBW }} \nc{\rbws}{\rm{RBWs }} \nc{\rcot}{\mrm{cot}}
\nc{\rest}{\rm{controlled}\xspace}
\nc{\rdef}{\mrm{def}} \nc{\rdiv}{{\rm div}} \nc{\rtf}{{\rm tf}}
\nc{\rtor}{{\rm tor}} \nc{\res}{\mrm{res}} \nc{\SL}{\mrm{SL}}
\nc{\Spec}{\mrm{Spec}} \nc{\tor}{\mrm{tor}} \nc{\Tr}{\mrm{Tr}}
\nc{\mtr}{\mrm{sk}}

%%%%%%%%%%%%%%%%%% bold face
\nc{\ab}{\mathbf{Ab}} \nc{\Alg}{\mathbf{Alg}}
\nc{\Algo}{\mathbf{Alg}^0} \nc{\Bax}{\mathbf{Bax}}
\nc{\Baxo}{\mathbf{Bax}^0} \nc{\RB}{\mathbf{RB}}
\nc{\RBo}{\mathbf{RB}^0} \nc{\BRB}{\mathbf{RB}}
\nc{\Dend}{\mathbf{DD}} \nc{\bfk}{{\bf k}} \nc{\bfone}{{\bf 1}}
\nc{\base}[1]{{a_{#1}}} \nc{\detail}{\marginpar{\bf More detail}
    \noindent{\bf Need more detail!}
    \svp}
\nc{\Diff}{\mathbf{Diff}} \nc{\gap}{\marginpar{\bf
Incomplete}\noindent{\bf Incomplete!!}
    \svp}
\nc{\FMod}{\mathbf{FMod}} \nc{\mset}{\mathbf{MSet}}
\nc{\rb}{\mathrm{RB}} \nc{\Int}{\mathbf{Int}}
\nc{\Mon}{\mathbf{Mon}}
%\nc{\remark}{\noindent{\bf Remark: }}
\nc{\remarks}{\noindent{\bf Remarks: }}
\nc{\OS}{\mathbf{OS}} %free operated semigroup
\nc{\Rep}{\mathbf{Rep}}
\nc{\Rings}{\mathbf{Rings}} \nc{\Sets}{\mathbf{Sets}}
\nc{\DT}{\mathbf{DT}}

%%%%%%%%%%%%%%%%%%%Bbb fonts
\nc{\BA}{{\mathbb A}} \nc{\CC}{{\mathbb C}} \nc{\DD}{{\mathbb D}}
\nc{\EE}{{\mathbb E}} \nc{\FF}{{\mathbb F}} \nc{\GG}{{\mathbb G}}
\nc{\HH}{{\mathbb H}} \nc{\LL}{{\mathbb L}} \nc{\NN}{{\mathbb N}}
\nc{\QQ}{{\mathbb Q}} \nc{\RR}{{\mathbb R}} \nc{\BS}{{\mathbb{S}}} \nc{\TT}{{\mathbb T}}
\nc{\VV}{{\mathbb V}} \nc{\ZZ}{{\mathbb Z}}

%%%%%%%%%%%%%%%%%%% cal fonts

\nc{\calao}{{\mathcal A}} \nc{\cala}{{\mathcal A}}
\nc{\calc}{{\mathcal C}} \nc{\cald}{{\mathcal D}}
\nc{\cale}{{\mathcal E}} \nc{\calf}{{\mathcal F}}
\nc{\calfr}{{{\mathcal F}^{\,r}}} \nc{\calfo}{{\mathcal F}^0}
\nc{\calfro}{{\mathcal F}^{\,r,0}} \nc{\oF}{\overline{F}}
\nc{\calg}{{\mathcal G}} \nc{\calh}{{\mathcal H}}
\nc{\cali}{{\mathcal I}} \nc{\calj}{{\mathcal J}}
\nc{\call}{{\mathcal L}} \nc{\calm}{{\mathcal M}}
\nc{\caln}{{\mathcal N}} \nc{\calo}{{\mathcal O}}
\nc{\calp}{{\mathcal P}} \nc{\calq}{{\mathcal Q}} \nc{\calr}{{\mathcal R}}
\nc{\calt}{{\mathcal T}} \nc{\caltr}{{\mathcal T}^{\,r}}
\nc{\calu}{{\mathcal U}} \nc{\calv}{{\mathcal V}}
\nc{\calw}{{\mathcal W}} \nc{\calx}{{\mathcal X}}
\nc{\CA}{\mathcal{A}}

%%%%%%%%%%%%%%%%%%  frak fonts
\nc{\fraka}{{\mathfrak a}} \nc{\frakB}{{\mathfrak B}}
\nc{\frakb}{{\mathfrak b}} \nc{\frakd}{{\mathfrak d}}
\nc{\oD}{\overline{D}}
\nc{\frakF}{{\mathfrak F}} \nc{\frakg}{{\mathfrak g}}
\nc{\frakm}{{\mathfrak m}} \nc{\frakM}{{\mathfrak M}}
\nc{\frakMo}{{\mathfrak M}^0} \nc{\frakp}{{\mathfrak p}}
\nc{\frakS}{{\mathfrak S}} \nc{\frakSo}{{\mathfrak S}^0}
\nc{\fraks}{{\mathfrak s}} \nc{\os}{\overline{\fraks}}
\nc{\frakT}{{\mathfrak T}}
\nc{\oT}{\overline{T}}
%\nc{\frakx}{{\mathfrak x}}
\nc{\frakX}{{\mathfrak X}} \nc{\frakXo}{{\mathfrak X}^0}
\nc{\frakx}{{\mathbf x}}
%\nc{\frakTxo}{{\frakTx}^0}
\nc{\frakTx}{\frakT}      %All rooted trees, correspond to \ncsha(X)
\nc{\frakTa}{\frakT^a}        % rooted trees for \ncsha(A)
\nc{\frakTxo}{\frakTx^0}   % rooted trees for \ncshao(X)
\nc{\caltao}{\calt^{a,0}}   % rooted trees for \ncshao(A)
\nc{\ox}{\overline{\frakx}} \nc{\fraky}{{\mathfrak y}}
\nc{\frakz}{{\mathfrak z}} \nc{\oX}{\overline{X}}

\font\cyr=wncyr10

\nc{\redtext}[1]{\textcolor{red}{#1}}
%\nc{\li}[1]{\textcolor{red}{#1}}

%%%%%%%%%%%%%%%%%%%%%%%%%%%%%%%%%%%%%%%%%%%%%%%%%%%%%%%%%%%%%%%%%%

%\begin{document}
\title{Classification of compatible left-symmetric conformal algebraic structures on the Lie conformal algebra $\mathcal{W}(a,b)$}

\author{Deng Liu}
\address{College of Science, Zhejiang Agriculture and Forestry University,
Hangzhou, 311300, P.R.China}
\email{17816898206@163.com}

\author{Yanyong Hong}
\address{College of Science, Zhejiang Agriculture and Forestry University,
Hangzhou, 311300, P.R.China}
\email{hongyanyong2008@yahoo.com}

\author{Hao Zhou}
\address{College of Science, Zhejiang Agriculture and Forestry University,
Hangzhou, 311300, P.R.China}
\email{haotryagain@sina.com}

\author{Nuan Zhang}
\address{College of Science, Zhejiang Agriculture and Forestry University,
Hangzhou, 311300, P.R.China}
\email{mariaj77@163.com}

\subjclass[2010]{17A30, 17D25, 17A60, 17B69}
\keywords{Lie conformal algebra, left-symmetric conformal algebra, coefficient algebra, left-symmetric algebra}

\begin{abstract}
In this paper, under some natural condition,  a complete classification of compatible left-symmetric conformal algebraic structures on the Lie conformal algebra $\mathcal{W}(a,b)$ is presented. Moreover, applying this result, we obtain a class of compatible left-symmetric algebraic structures on the coefficient algebra of $\mathcal{W}(a,b)$.

\end{abstract}

\maketitle

%\footnote{The second author is the corresponding author.}%
\section{Introduction}
Left-symmetric algebras are a class of Lie-admissible algebras whose commutators are Lie algebras. They originated from the study of convex homogeneous cones \cite{V}, affine manifolds and affine structures on Lie
groups \cite{Ko}, deformation of associative algebras \cite{G} and so on. Novikov
algebras are left-symmetric algebras whose right multiplications are commutative. It was essentially stated in \cite{GD} that they correspond to
certain Hamiltonian operators. Such algebraic structures also appeared
in \cite{BN} from the point of view of Poisson structures of
hydrodynamic type. A
survey about left-symmetric algebras was given in \cite{Bu1}, which showed that they play an important role in many fields in mathematics and mathematical physics such as vector fields, rooted tree algebras, words in two letters, operad theory, vertex algebras,  deformation complexes of algebras, affine manifolds, convex homogeneous cones,  left-invariant affine structures on Lie groups (see \cite{Bu1} and the references therein). An important problem in the study of left-symmetric algebras is to determine all compatible left-symmetric algebraic structures on a Lie algebra. It is known that when the character of the field is $0$, all finite-dimensional semisimple Lie algebras have no compatible left-symmetric algebraic structures. Moreover, there are a lot of results in the study of infinite-dimensional cases. For example, compatible left-symmetric algebraic structures on
Witt algebra and Virasoro algebra were studied in \cite{Ch, KCB, Ku1, Ku2, Os, TB, LGB, Xu} and so on. In particular,
a complete classification of graded compatible left-symmetric algebraic structures on Witt algebra and Virasoro algebra was given in \cite{KCB}. The case of super-Virasoro algebra can refer to \cite{KB}. Moreover, some similar classifications of compatible left-symmetric algebraic structures on the twisted Heisenberg-Virasoro algebra and $W$-algebra
$W(2,2)$ were given in \cite{CL1,CL2}. It should be pointed out that these researches heavily depend on the result of a classification of compatible graded left-symmetric algebraic structures on Witt algebra in \cite{KCB}.

In fact, these infinite-dimensional Lie algebras mentioned above are all formal distribution Lie algebras (see \cite{K1}). Moreover,
a formal distribution Lie algebra can be seen as a Lie conformal algebra (see \cite{K1}). The notion of Lie conformal algebras, introduced by
V.Kac, gives an
axiomatic description of the operator product expansion (or rather
its Fourier transform) of chiral fields in conformal field theory. It is a wonderful tool to study vertex algebras (see \cite{K1}). Moreover, Lie conformal algebras have many applications in the theory of infinite-dimensional Lie algebras satisfying the locality property in \cite{K} and Hamiltonian formalism in the theory of nonlinear evolution equations (see \cite{BDK}).
The structure theory (see \cite{DK1}), cohomology theory (see \cite{BKV}) and representation theory (see \cite{CK1})
of finite Lie conformal algebras
have been well developed. It is known that the category of Lie conformal algebras is almost equivalent to the category of formal distribution Lie algebras (see \cite{K1}). Witt algebra, Virasoro algebra, twisted Heisenberg-Virasoro algebra, and $W(2,2)$ can be regarded as coefficient algebras of the corresponding Lie conformal algebras. In addition, to  investigate whether there exist compatible left-symmetric algebraic structures on formal distribution Lie algebras, a definition of left-symmetric conformal algebras was introduced in \cite{HL}, which can be used to construct vertex algebras. It was also shown that the commutator of a left-symmetric conformal algebra in the conformal sense is a Lie conformal algebra, and if a Lie conformal algebra has a compatible left-symmetric conformal algebraic structure, its coefficient algebra has the corresponding compatible left-symmetric algebraic structure. Motivated by the study on compatible left-symmetric algebraic structures on infinite-dimensional Lie algebras, a natural idea is to study compatible left-symmetric conformal algebraic structures on some Lie conformal algebras corresponding to some important infinite-dimensional Lie algebras. Because we can obtain some new compatible left-symmetric algebraic structures on the corresponding Lie algebras, this study is meaningful. Since the Lie conformal algebras corresponding to the twisted Heisenberg-Virasoro algebra and $W(2,2)$ without central extensions are of rank $2$,
our studying object in this paper is a class of more generalized Lie conformal algebras $\mathcal{W}(a,b)$ of rank $2$.
$\mathcal{W}(1,0)$ is just the Heisenberg-Virasoro Lie conformal algebra. Its structure, representation and cohomology were studied in \cite{LY, WY, YW1}. $\mathcal{W}(2,0)$ is just the $W(2,2)$ Lie conformal algebra. The study of its structure and representation can refer to \cite{YW2, WY}. In this paper, our aim is to study compatible left-symmetric conformal algebraic structures on $\mathcal{W}(a,b)=\mathbb{C}[\partial]L\oplus\mathbb{C}[\partial]W$. According to
that in the study of compatible left-symmetric algebraic structures on the twisted Heisenberg-Virasoro algebra and $W(2,2)$,
\cite{CL1,CL2} assumed that $\oplus_{i=-\infty}^{+\infty}\mathbb{C}L_i$ is a graded left-symmetric algebra, in this paper, we also assume that $\mathbb{C}[\partial]L$ is a left-symmetric conformal algebra in $\mathcal{W}(a,b)$.
Under this condition, we present a complete classification of compatible left-symmetric conformal algebraic structures on $\mathcal{W}(a,b)$. Moreover, applying the obtained result, we can get some compatible left-symmetric algebraic structures on the coefficient algebra of $\mathcal{W}(a,b)$. It should be pointed out that it is hard to give a complete classification of left-symmetric conformal algebras of rank 2 which are free as $\mathbb{C}[\partial]$-modules.

Throughout this paper, denote by $\mathbb{C}$ the field of complex
numbers; $\mathbb{N}$ the set of natural numbers, i.e.
$\mathbb{N}=\{0, 1, 2,\cdots\}$; $\mathbb{Z}$ the set of integer
numbers.
Moreover, if $A$ is a vector space, then the space of polynomials of $\lambda$ with coefficients in $A$ is denoted by $A[\lambda]$.

\section{Preliminaries}
In this section, we will introduce the definitions and some relevant results about left-symmetric conformal algebras and Lie conformal algebras. These facts can refer to \cite{K1} and \cite{HL}.

\begin{defi} \label{21}{\rm
A {\bf conformal algebra} $R$ is a $\mathbb{C}[\partial]$-module
endowed with a $\mathbb{C}$-bilinear map $R\times R\rightarrow
 R[\lambda]$ denoted by $a\times b\rightarrow
a_{\lambda} b$ satisfying
\begin{eqnarray}
(\partial a)_{\lambda}b=-\lambda a_{\lambda}b,  \quad
a_{\lambda}(\partial b)=(\partial+\lambda)a_{\lambda}b.\end{eqnarray}

A {\bf Lie conformal algebra} $R$ is a conformal algebra with a $\mathbb{C}$-bilinear
map $[\cdot_\lambda \cdot]: R\times R\rightarrow  R[\lambda]$ satisfying
\begin{eqnarray*}
&&[a_\lambda b]=-[b_{-\lambda-\partial}a],~~~~\text{(skew-symmetry)}\\
&&[a_\lambda[b_\mu c]]=[[a_\lambda b]_{\lambda+\mu} c]+[b_\mu[a_\lambda c]],~~~~~~\text{(Jacobi identity)}
\end{eqnarray*}
for $a$, $b$, $c\in R$.

A {\bf left-symmetric conformal algebra} $R$ is a conformal algebra with a $\mathbb{C}$-bilinear
map $\cdot_\lambda \cdot: R\times R\rightarrow R[\lambda]$ satisfying
\begin{eqnarray}
(a_{\lambda}b)_{\lambda+\mu}c-a_{\lambda}(b_\mu
c)=(b_{\mu}a)_{\lambda+\mu}c-b_\mu(a_\lambda c),
\end{eqnarray}
for $a$, $b$, $c\in R$.}
\end{defi}

Here, $a_\lambda b=\sum_{n=0}^\infty\frac{\lambda^{n}}{n!}(a_{(n)}b)$, where $a_{(n)}b$ is called {\bf the $n$-th product }of $a$ and $b$. The notions of a homomorphism, ideal and subalgebra of a conformal algebra are defined as usual. A conformal algebra is called {\bf finite}, if it is finitely generated as a $\mathbb{C}[\partial]$-module. The {\bf rank} of a conformal
algebra is its rank as a $\mathbb{C}[\partial]$-module.

By Definition \ref{21}, we can easily obtain
\begin{pro}(see Proposition 2.5 in \cite{HL})
Suppose that $(A, \cdot_\lambda \cdot)$ is a left-symmetric conformal algebra. Then the following
$\lambda$-bracket
\begin{equation}\label{3}[a_\lambda b]=a_\lambda
b-b_{-\lambda-\partial}a,~~~~~~a, ~~b \in A,\end{equation}
can define a Lie conformal algebra $\mathfrak{g}(A)$.
\end{pro}

\begin{rmk}
In the above proposition, $\mathfrak{g}(A)$ is called the
{\bf sub-adjacent Lie conformal algebra} of $A$ and $A$ is also called {\bf a
compatible left-symmetric conformal algebraic structure} on the Lie
conformal algebra $\mathfrak{g}(A)$.
\end{rmk}

Assume that $R$ is a (Lie or left-symmetric) conformal algebra. Let Coeff$(R)$ be the quotient
of the vector space with basis $a_n$ $(a\in R, n\in\mathbb{Z})$ by
the subspace spanned over $\mathbb{C}$ by
elements:\\
$$(\alpha a)_n-\alpha a_n,~~(a+b)_n-a_n-b_n,~~(\partial
a)_n+na_{n-1},~~~\text{where}~~a,~~b\in R,~~\alpha\in \mathbb{C},~~n\in
\mathbb{Z}.$$ We can define the following operation on Coeff$(R)$:
\begin{equation}\label{106}
a_m\cdot b_n=\sum_{j\in \mathbb{N}}\left(\begin{array}{ccc}
m\\j\end{array}\right)(a_{(j)}b)_{m+n-j}.\end{equation} Then
Coeff$(R)$ is a (Lie or left-symmetric) algebra (see \cite{K1}).

\begin{pro}\label{pp1}(see Theorem 4.7 in \cite{HL})
If $A$ is a compatible left-symmetric conformal algebra on a Lie conformal algebra $R$, then Coeff$(A)$ is a compatible left-symmetric algebra on the Lie algebra Coeff$(R)$.
\end{pro}

\begin{ex}\label{r1}
The Virasoro Lie conformal algebra $Vir=\mathbb{C}[\partial]L$ is of the following $\lambda$-bracket
\begin{align*}
[L_\lambda L]=(\partial+2\lambda)L.
\end{align*}

Any compatible left-symmetric conformal algebra over $Vir$ is of the following form (see Theorem 3.2 in \cite{HL}):
\begin{align*}
Vir=\mathbb{C}[\partial]L,~~~~~L_\lambda L=(\partial+\lambda+c)L,
\end{align*}
where $c\in \mathbb{C}$.

By the definition of a coefficient algebra, Coeff$(Vir)$ is isomorphic to Witt algebra. Moreover, by Proposition \ref{pp1},
there are the following compatible left-symmetric algebraic structures on Witt algebra:
\begin{align*}
Coeff(Vir)=\bigoplus_{i=-\infty}^{+\infty}\mathbb{C}L_i,~~~~~L_i\circ L_j=cL_{i+j+1}-(j+1)L_{i+j}.
\end{align*}
The detailed explanation can be referred to Remark 3.3 in \cite{HL}.
\end{ex}

In this paper, we will focus on a class of Lie conformal algebras of rank 2.
\begin{defi}
$\mathcal{W}(a,b)=\mathbb{C}[\partial]L\oplus\mathbb{C}[\partial]W$ is a Lie conformal algebra with the following $\lambda$-brackets
\begin{align*}
[L_\lambda L]=(\partial+2\lambda)L,~~[L_\lambda W]=(\partial+a\lambda+b)W,~~[W_\lambda W]=0,
\end{align*}
where $a$, $b\in \mathbb{C}$.
\end{defi}

According to the definition of a coefficient algebra, Coeff$(\mathcal{W}(a,b))=\bigoplus_{i=-\infty}^{+\infty}\mathbb{C}L_i\oplus\bigoplus_{i=-\infty}^{+\infty}\mathbb{C}W_i$
is isomorphic to $\bigoplus_{i=-\infty}^{+\infty}\mathbb{C}L_i\oplus\bigoplus_{i=-\infty}^{+\infty}\mathbb{C}W_i$ with Lie brackets as follows:
\begin{align}
\label{ww1}[L_m,L_n]=(m-n)L_{m+n},~~~~~~[W_m,W_n]=0.\\
\label{ww2}[L_m,W_n]=((m+1)(a-1)-(n+1))W_{m+n}+bW_{m+n+1},
\end{align}
through the isomorphism $\varphi(L_i)=L_{i+1}$ and $\varphi(W_i)=W_{i+1}$ for all $i\in \mathbb{Z}$. Therefore,  in the sequel, we always take Coeff$(\mathcal{W}(a,b))$ as a Lie algebra with the Lie brackets given by (\ref{ww1}) and (\ref{ww2}).
\begin{rmk}
It should be pointed out that Coeff$(\mathcal{W}(1,0))$ is isomorphic to the twisted Heisenberg-Virasoro Lie algebra without central extensions and Coeff$(\mathcal{W}(2,0))$ is just the $W(2,2)$ without central extensions.
\end{rmk}

\section{Compatible left-symmetric conformal algebraic structures on $\mathcal{W}(a,b)$}
In this section, we plan to give a classification of compatible left-symmetric conformal algebraic structures on $\mathcal{W}(a,b)$. Since there was an assumption that $\bigoplus_{i=-\infty}^{+\infty}\mathbb{C}L_i$ is a left-symmetric subalgebra in \cite{KB,CL1,CL2}, in this paper, we also assume that $\mathbb{C}[\partial]L$ is a left-symmetric conformal subalgebra in $\mathcal{W}(a,b)$.

Therefore, by the hypothesis above, we can assume
\begin{align}
\label{e1}L_\lambda L=f(\lambda,\partial)L,\\
L_\lambda W=g_1(\lambda,\partial)L+g_2(\lambda,\partial)W,\\
W_\lambda L=h_1(\lambda,\partial)L+h_2(\lambda,\partial)W,\\
\label{e2}W_\lambda W=k_1(\lambda,\partial)L+k_2(\lambda,\partial)W,
\end{align}
where $f(\lambda,\partial)$, $g_1(\lambda,\partial)$, $g_2(\lambda,\partial)$, $h_1(\lambda,\partial)$, $h_2(\lambda,\partial)$, $k_1(\lambda,\partial)$ and $k_2(\lambda,\partial)\in \mathbb{C}[\lambda,\partial]$.

According to that $\mathbb{C}[\partial]L$ is a left-symmetric conformal subalgebra, $[L_\lambda L]=L_\lambda L-L_{-\lambda-\partial}L$ and by Example \ref{r1}, we get
$f(\lambda,\partial)=\partial+\lambda+c$, where $c\in \mathbb{C}$. By the identity of left-symmetric conformal algebras and the compatible condition, the algebra defined by (\ref{e1})-(\ref{e2}) is a compatible left-symmetric conformal algebra on $\mathcal{W}(a,b)$ if and only if the following identities hold:
\begin{align}
[L_\lambda W]_{\lambda+\mu}L=L_\lambda(W_\mu L)-W_\mu(L_\lambda L),\\
[L_\lambda L]_{\lambda+\mu}W=L_\lambda(L_\mu W)-L_\mu(L_\lambda W),\\
[L_\lambda W]_{\lambda+\mu}W=L_\lambda(W_\mu W)-W_\mu(L_\lambda W),\\
[W_\lambda W]_{\lambda+\mu}L=W_\lambda(W_\mu L)-W_\mu(W_\lambda L),\\
[W_\lambda W]_{\lambda+\mu}W=W_\lambda(W_\mu W)-W_\mu(W_\lambda W),\\
L_\lambda W-W_{-\lambda-\partial}L=[L_\lambda W],\\
W_\lambda W-W_{-\lambda-\partial}W=[W_\lambda W]=0.
\end{align}
Substituting (\ref{e1})-(\ref{e2}) into these identities,  the algebra defined by (\ref{e1})-(\ref{e2}) is a compatible left-symmetric conformal algebra on $\mathcal{W}(a,b)$ if and only if
 $g_1(\lambda,\partial)$, $g_2(\lambda,\partial)$, $h_1(\lambda,\partial)$, $h_2(\lambda,\partial)$, $k_1(\lambda,\partial)$ and $k_2(\lambda,\partial)$ satisfy the following equations
\begin{align}
&&(-\lambda-\mu+a\lambda+b)h_1(\lambda+\mu,\partial)=h_1(\mu,\lambda+\partial)(\partial+\lambda+c)\nonumber\\
\label{f1}&&+h_2(\mu,\lambda+\partial)g_1(\lambda,\partial)
-(\partial+\mu+\lambda+c)h_1(\mu,\partial),\\
\label{f2}&&(-\lambda-\mu+a\lambda+b)h_2(\lambda+\mu,\partial)=h_2(\mu,\lambda+\partial)g_2(\lambda,\partial)
-(\partial+\mu+\lambda+c)h_2(\mu,\partial),\\
&&(\lambda-\mu)g_1(\lambda+\mu,\partial)=g_1(\mu,\lambda+\partial)(\partial+\lambda+c)+g_2(\mu,\lambda+\partial)g_1(\lambda,\partial)\nonumber\\
\label{f3}&&-g_1(\lambda,\mu+\partial)(\partial+\mu+c)-g_2(\lambda,\mu+\partial)g_1(\mu,\partial),\\
\label{f4}&&(\lambda-\mu)g_2(\lambda+\mu,\partial)=g_2(\mu,\lambda+\partial)g_2(\lambda,\partial)
-g_2(\lambda,\mu+\partial)g_2(\mu,\partial),\\
&&(-\lambda-\mu+a\lambda+b)k_1(\lambda+\mu,\partial)=k_1(\mu,\lambda+\partial)(\partial+\lambda+c)+k_2(\mu,\lambda+\partial)g_1(\lambda,\partial)\nonumber\\
\label{f5}&&-g_1(\lambda,\mu+\partial)h_1(\mu,\partial)-g_2(\lambda,\mu+\partial)k_1(\mu,\partial),\\
&&(-\lambda-\mu+a\lambda+b)k_2(\lambda+\mu,\partial)=k_2(\mu,\lambda+\partial)g_2(\lambda,\partial)
-g_1(\lambda,\mu+\partial)h_2(\mu,\partial)\nonumber\\
\label{f6}&&-g_2(\lambda,\mu+\partial)k_2(\mu,\partial),\end{align}
\begin{align}\label{f7}&&h_1(\mu,\lambda+\partial)h_1(\lambda,\partial)+h_2(\mu,\lambda+\partial)k_1(\lambda,\partial)=h_1(\lambda,\mu+\partial)h_1(\mu,\partial)
+h_2(\lambda,\mu+\partial)k_1(\mu,\partial),\\
\label{f8}&&h_1(\mu,\lambda+\partial)h_2(\lambda,\partial)+h_2(\mu,\lambda+\partial)k_2(\lambda,\partial)=h_1(\lambda,\mu+\partial)h_2(\mu,\partial)
+h_2(\lambda,\mu+\partial)k_2(\mu,\partial),\\
\label{f9}&&k_1(\mu,\lambda+\partial)h_1(\lambda,\partial)+k_2(\mu,\lambda+\partial)k_1(\lambda,\partial)=k_1(\lambda,\mu+\partial)h_1(\mu,\partial)
+k_2(\lambda,\mu+\partial)k_1(\mu,\partial),\\
\label{f10}&&k_1(\mu,\lambda+\partial)h_2(\lambda,\partial)+k_2(\mu,\lambda+\partial)k_2(\lambda,\partial)=k_1(\lambda,\mu+\partial)h_2(\mu,\partial)
+k_2(\lambda,\mu+\partial)k_2(\mu,\partial),\\
\label{f11}&&g_1(\lambda,\partial)-h_1(-\lambda-\partial,\partial)=0,\\
\label{f12}&&g_2(\lambda,\partial)-h_2(-\lambda-\partial,\partial)=\partial+a\lambda+b,\\
\label{f13}&&k_1(\lambda,\partial)=k_1(-\lambda-\partial,\partial),\\
\label{f14}&&k_2(\lambda,\partial)=k_2(-\lambda-\partial,\partial).
\end{align}

\begin{lem}\label{lem1}
For $g_2(\lambda,\partial)$ and $h_2(\lambda,\partial)$, there are the following three cases: \\
(A) $g_2(\lambda,\partial)=\partial+a\lambda+b$, $h_2(\lambda,\partial)=0$;\\
(B) $g_2(\lambda,\partial)=\partial+a\lambda+b+c$, $h_2(\lambda,\partial)=c$, where $c\neq 0$;\\
(C) $g_2(\lambda,\partial)=\partial+(a-1)\lambda+b+c$, $h_2(\lambda,\partial)=\partial+\lambda+c$.
\end{lem}
\begin{proof}
First, let us determine $g_2(\lambda,\partial)$. By comparing the degree of $\lambda$ in (\ref{f4}),
we can obtain that the degree of $\partial$ in $g_2(\lambda,\partial)$ is smaller than $2$. Therefore, assume that $g_2(\lambda,\partial)=g_{21}(\lambda)\partial+g_{22}(\lambda)$, where $g_{21}(\lambda)$, $g_{22}(\lambda)\in\mathbb{C}[\lambda]$. Taking it into (\ref{f4}), one can get
\begin{align}
\label{f15}(\lambda-\mu)(g_{21}(\lambda+\mu)\partial+g_{22}(\lambda+\mu))=g_{21}(\lambda)g_{21}(\mu)(\lambda-\mu)\partial+g_{21}(\mu)g_{22}(\lambda)\lambda-g_{21}(\lambda)g_{22}(\mu)\mu.
\end{align}
So, we have $g_{21}(\lambda+\mu)=g_{21}(\lambda)g_{21}(\mu)$. Therefore, $g_{21}(\lambda)=0$ or $g_{21}(\lambda)=1$.
When $g_{21}(\lambda)=0$, by (\ref{f15}), we can obtain $g_{22}(\lambda)=0$. When $g_{21}(\lambda)=1$, (\ref{f15}) can be simplified
into
\begin{align}
\label{f16}(\lambda-\mu)g_{22}(\lambda+\mu)=g_{22}(\lambda)\lambda-g_{22}(\mu)\mu. \end{align}
It is obvious that the degree of $g_{22}(\lambda)$ is smaller than $2$. Therefore we can assume $g_{22}(\lambda)=f+e\lambda$, where $e$, $f\in \mathbb{C}$.
Then (\ref{f16}) naturally holds. By the above discussion, we can get that $g_2(\lambda,\partial)=0$ or $g_2(\lambda,\partial)=\partial+e\lambda+f$. When $g_2(\lambda,\partial)=0$, by (\ref{f12}), we have $h_2(\lambda,\partial)=(a-1)\partial+a\lambda-b$. Then according to (\ref{f2}), we get $(-\lambda-\mu+a\lambda+b)((a-1)\partial+a\lambda+a\mu-b)=
-(\partial+\mu+\lambda+c)((a-1)\partial+a\mu-b)$. By comparing the degree of $\partial$ in the above equality,
one can know that this case does not hold. Therefore, we have $g_2(\lambda,\partial)=\partial+e\lambda+f$.

By (\ref{f12}), $h_2(\lambda,\partial)=(a-e)(\lambda+\partial)+f-b$. Putting it into (\ref{f2}), one can get
\begin{align}
&&(\partial+(e-a+1)\lambda+\mu+f-b)((a-e)(\partial+\lambda+\mu)+f-b)\nonumber\\
\label{f17}&&=(\partial+\mu+\lambda+c)((a-e)(\partial+\mu)+f-b).
\end{align}
By comparing the coefficients of $\lambda\partial$ in (\ref{f17}), we obtain $(e-a+1)(a-e)=0$. Therefore, $e=a$ or $e=a-1$.
When $e=a$, by (\ref{f17}), one can have $f=b$ or $f=b+c$; when $e=a-1$, it can be directly obtained from (\ref{f17}) that
$f=b+c$. Therefore, we get this lemma.\end{proof}

Next, we only need to discuss the three cases in Lemma \ref{lem1}.

\begin{lem}\label{lem2}
In Case (A) in Lemma \ref{lem1}, we can get the following cases:\\
(A1) $g_2(\lambda,\partial)=\partial+a\lambda+b$, $h_2(\lambda,\partial)=h_1(\lambda,\partial)=g_1(\lambda,\partial)=k_1(\lambda,\partial)=k_2(\lambda,\partial)=0$;\\
(A2) When $a=1$ and $c=2b$, $g_2(\lambda,\partial)=\partial+\lambda+b$, $k_1(\lambda,\partial)=k_1$, $h_2(\lambda,\partial)=h_1(\lambda,\partial)=g_1(\lambda,\partial)=k_2(\lambda,\partial)=0$, where $k_1\in \mathbb{C}\backslash\{0\}$;\\
(A3) When $a=1$ and $b=0$, $g_2(\lambda,\partial)=\partial+\lambda$, $k_2(\lambda,\partial)=k_2$, $h_2(\lambda,\partial)=h_1(\lambda,\partial)=g_1(\lambda,\partial)=k_1(\lambda,\partial)=0$,
where $k_2\neq 0$;\\
(A4) When $a=1$ and $c=b=0$, $g_2(\lambda,\partial)=\partial+\lambda$, $k_1(\lambda,\partial)=k_1$, $k_2(\lambda,\partial)=k_2$, $h_2(\lambda,\partial)=h_1(\lambda,\partial)=g_1(\lambda,\partial)=0$, where $k_1$, $k_2\in \mathbb{C}\backslash\{0\}$;\\
(A5) When $a=1$ and $b=0$, $g_2(\lambda,\partial)=\partial+\lambda$, $h_1(\lambda,\partial)=g_1(\lambda,\partial)=h_1$, $k_2(\lambda,\partial)=k_2$, $h_2(\lambda,\partial)=0$, $k_1(\lambda,\partial)=\frac{h_1(h_1-k_2)}{c}$, where $h_1$, $c\neq 0$, $k_2\in \mathbb{C}$;\\
(A6) When $a=1$ and $b=0$, $g_2(\lambda,\partial)=\partial+\lambda$, $h_1(\lambda,\partial)=g_1(\lambda,\partial)=h_1$, $k_2(\lambda,\partial)=h_1$, $h_2(\lambda,\partial)=0$, $k_1(\lambda,\partial)=k_1$, where $h_1$, $k_1\neq 0$, $c=0$.
\end{lem}
\begin{proof}
According to Case (A), $g_2(\lambda,\partial)=\partial+a\lambda+b$, $h_2(\lambda,\partial)=0$. By (\ref{f7}), one can get $h_1(\mu,\lambda+\partial)h_1(\lambda,\partial)=h_1(\lambda,\mu+\partial)h_1(\mu,\partial)$. By comparing the degree of $\lambda$ in the above equality, one can know that the degree of $\partial$ in $h_1(\lambda,\partial)$ is
$0$. Therefore, assume $h_1(\lambda,\partial)=h_1(\lambda)$, where $h_1(\lambda)\in \mathbb{C}[\lambda]$. Then
$g_1(\lambda,\partial)=h_1(-\lambda-\partial)$ by (\ref{f11}). Substituting it into (\ref{f1}), one can have
\begin{align}
\label{n1}((a-1)\lambda-\mu+b)h_1(\lambda+\mu)=-\mu h_1(\mu).
\end{align}
Consequently, when $a\neq 1$ or $b\neq 0$, $h_1(\lambda)=0$, i.e. $h_1(\lambda,\partial)=g_1(\lambda,\partial)=0$;
when  $a=1$ and $b=0$, $h_1(\mu)=h_1$, where $h_1\in \mathbb{C}$, i.e. $h_1(\lambda,\partial)=g_1(\lambda,\partial)=h_1$.
On the other hand, similarly, by (\ref{f10}), we can assume that $k_2(\lambda,\partial)=k_2(\lambda)$, where $k_2(\lambda)\in \mathbb{C}[\lambda]$. By (\ref{f14}), $k_2(\lambda)=k_2(-\lambda-\partial)$. Therefore, $k_2(\lambda)=k_2$, where $k_2\in \mathbb{C}$. Then by (\ref{f6}), we can get $((a-1)\lambda+b)k_2=0$. As a result, when $a\neq 1$ or $b\neq 0$, $k_2=0$,
i.e. $k_2(\lambda,\partial)=0$; when $a=1$ and $b=0$, $k_2(\lambda,\partial)=k_2$. By (\ref{f5}) and (\ref{f9}),
one can get
\begin{align}
\label{pn1}((a-1)\lambda-\mu+b)k_1(\lambda+\mu,\partial)=(\partial+\lambda+c)k_1(\mu,\lambda+\partial)+h_1(k_2-h_1)-(\partial+\mu+a\lambda+b)k_1(\mu,\partial),\\
\label{pn2}h_1(k_1(\mu,\lambda+\partial)-k_1(\lambda,\mu+\partial))=k_2(k_1(\mu,\partial)-k_1(\lambda,\partial)),
\end{align}
where $h_1=k_2=0$ if $a\neq 1$ or $b\neq 0$. Then we discuss (\ref{pn1}) and (\ref{pn2}) in three subcases.

{\bf Subcase a1}: $h_1=0$, $k_2=0$, i.e. $h_1(\lambda,\partial)=k_2(\lambda,\partial)=0$. By (\ref{pn1}),
we have
\begin{align}
\label{pn3}((a-1)\lambda-\mu+b)k_1(\lambda+\mu,\partial)=(\partial+\lambda+c)k_1(\mu,\lambda+\partial)-(\partial+\mu+a\lambda+b)k_1(\mu,\partial).
\end{align}
Setting $\lambda=0$ in (\ref{pn3}), one can obtain $(c-2b)k_1(\mu,\partial)=0$. Therefore, when $c\neq 2b$, $k_1(\mu,\partial)=0$. When $c=2b$, setting $\partial=0$ in (\ref{pn3}), we can get $k_1(\mu,\lambda)=\frac{((a-1)\lambda-\mu+b)k_1(\lambda+\mu,0)+(\mu+a\lambda+b)k_1(\mu,0)}{\lambda+c}$.
Assume $k_1(\mu,0)=k_1(\mu)$, where $k_1(\mu)\in \mathbb{C}[\mu]$. Then
$k_1(\lambda,\partial)=\frac{((a-1)\partial-\lambda+b)k_1(\lambda+\partial)+(\lambda+a\partial+b)k_1(\lambda)}{\partial+2b}$.
Taking it into
(\ref{pn3}), we have
\begin{align}
\label{ge6}(((a-1)\lambda-\mu+b)((a-1)\partial-\lambda-\mu+b)-(\partial+2b)((a-1)(\lambda+\partial)-\mu+b))k_1(\lambda+\mu+\partial)\nonumber\\
+(\partial+\mu+a\lambda+b)((a-1)\partial-\mu+b)k_1(\mu+\partial)\nonumber\\
=((a-a^2)\lambda\partial-a\lambda\mu-a\mu\partial-\mu^2+ab\lambda+ab\partial+b^2)k_1(\mu)\nonumber\\
-((a-1)\lambda-\mu+b)(a\partial+\lambda+\mu+b)k_1(\lambda+\mu).
\end{align}
If $a=1$, (\ref{ge6}) can be simplified into $k_1(\lambda+\mu+\partial)-k_1(\mu+\partial)=k_1(\lambda+\mu)-k_1(\mu)$.
Therefore, we can assume $k_1(\lambda)=e_0+e_1\lambda$, i.e $k_1(\lambda,\partial)=\frac{2e_0b+2be_1\lambda+(be_1+e_0)\partial}{\partial+2b}$, where $e_0$, $e_1\in \mathbb{C}$. Since $\partial+2b$ can divide $2e_0b+2be_1\lambda+(be_1+e_0)\partial$, we have $b=0$ or $e_1=0$. In both two cases, we all have $k_1(\lambda,\partial)=e_0$.
If $a\neq 1$, assume that $k_1(\lambda)=\sum_{i=0}^me_i\lambda^i$, where $e_m\neq 0$. If $m\geq 2$, by comparing the coefficients of $\lambda^{m+1}\partial$ and $\lambda^2\partial^m$ in (\ref{ge6}), we obtain
\begin{align}
\label{mq1}(a-1)(2a-2-m)=0,\\
\label{mq2}(a-1)(am-2m-1-\frac{m(m-1)}{2})=0.
\end{align}
By (\ref{mq1}), $a=1+\frac{m}{2}$. Putting it into (\ref{mq2}), we get $m=-2$. It contradicts to $m\geq 2$.
So, $m\leq 1$. Then we can set $k_1(\lambda)=e_0+e_1\lambda$, where $e_0$, $e_1\in \mathbb{C}$. According to $k_1(\lambda)=k_1(-\lambda)$, one can have
$e_1=0$. Therefore,  we get $k_1(\lambda,\partial)=\frac{e_0((2a-1)\partial+2b)}{\partial+2b}$.
Since $\partial+2b$ can divide $e_0((2a-1)\partial+2b)$, $e_0=0$ or $b=0$. When $e_0=0$, $k_1(\lambda,\partial)=0$.
When $b=0$, $k_1(\lambda,\partial)=e_0(2a-1)$. Taking it into (\ref{pn3}), one can get $2(a-1)(2a-1)e_0\lambda=0$.
Since $a\neq 1$, $(2a-1)e_0=0$, i.e. $k_1(\lambda,\partial)=0$. Therefore, according to the above discussion,
in this case, we have Case (A1) and Case (A2) in this lemma. It can be verified that in both two cases, (\ref{f1})-(\ref{f14}) hold.

{\bf Subcase a2}: $h_1=0$, $k_2\neq0$, i.e. $h_1(\lambda,\partial)=0$ and $k_2(\lambda,\partial)=k_2\neq0$. By
the above discussion, in this case, $a=1$ and $b=0$. By (\ref{pn2}), $k_1(\mu,\partial)=k_1(\lambda,\partial)$.
Therefore, $k_1(\lambda,\partial)=k_1(\partial)$, where $k_1(\partial)\in \mathbb{C}[\partial]$. Substituting it into (\ref{pn1}),
one can obtain
$(\partial+\lambda)k_1(\partial)=(\partial+\lambda+c)k_1(\lambda+\partial)$. Consequently,
when $c\neq 0$, $k_1(\lambda,\partial)=k_1(\partial)=0$; when $c=0$, $k_1(\lambda,\partial)=k_1$.
Therefore, according to the above discussion, in this case, we have Case (A3) and Case (A4).  It can be verified that in both two cases, (\ref{f1})-(\ref{f14}) hold.

{\bf Subcase a3}: $h_1\neq 0$. In this case, $a=1$, $b=0$. According to (\ref{pn2}), we get $k_1(\mu,\lambda+\partial)-k_1(\lambda,\mu+\partial)
=\frac{k_2}{h_1}(k_1(\mu,\partial)-k_1(\lambda,\partial))$. Taking $\mu=0$ in the above equality and setting $k_1(0,\partial)=k_1(\partial)$, one can obtain
\begin{align}
\label{pn4}(k_2-h_1)k_1(\lambda,\partial)=k_2k_1(\partial)-h_1k_1(\lambda+\partial).
\end{align}
It can be obtained from (\ref{pn4}) that when $k_2\neq h_1$, $k_1(\lambda,\partial)=\frac{k_2k_1(\partial)-h_1k_1(\lambda+\partial)}{k_2-h_1}$.
By $k_1(\lambda,\partial)=k_1(-\lambda-\partial,\partial)$, we have $k_1(\lambda+\partial)=k_1(-\lambda)$, i.e. $k_1(\lambda)=k_1$, where $k_1\in \mathbb{C}$. Therefore, when
$k_2\neq h_1$, $k_1(\lambda,\partial)=k_1$. Substituting it into (\ref{pn1}), we can obtain $ck_1=h_1(h_1-k_2)$.
Consequently, when $k_2\neq h_1$ and $c\neq 0$, $k_1(\lambda,\partial)=\frac{h_1(h_1-k_2)}{c}$. When
$k_2=h_1$, by (\ref{pn1}) and according to the computation in Subcase a1, we get $k_1(\lambda,\partial)=k_1$.
Taking it into (\ref{pn1}), we have $ck_1=0$. Therefore, when $c\neq 0$, $k_1(\lambda,\partial)=0$; when
$c=0$, $k_1(\lambda,\partial)=k_1$. Therefore, according to the above discussion, in this case, we have Case (A5) and Case (A6). It can be verified that in both two cases, (\ref{f1})-(\ref{f14}) hold.

This proof is completed.
\end{proof}
\begin{lem}\label{lem3}
In Case (B) in Lemma \ref{lem1}, when $b\neq 0$, we have the following results:\\
(B1) $h_1(\lambda,\partial)=g_1(\lambda,\partial)=k_1(\lambda,\partial)=k_2(\lambda,\partial)=0$, $g_2(\lambda,\partial)=\partial+a\lambda+b+c$ and $h_2(\lambda,\partial)=c$, where $c\neq0$;\\
(B2) When $a=1$ and $c=b$, $g_2(\lambda,\partial)=\partial+\lambda+2b$, $h_2(\lambda,\partial)=b$, $g_1(\lambda,\partial)=h_1(\lambda,\partial)=d$,
$k_1(\lambda,\partial)=-\frac{d^2}{b}$ and $k_2(\lambda,\partial)=-d$, where $d\in \mathbb{C}\backslash \{0\}$.
\end{lem}

\begin{proof}
In Case (B), $g_2(\lambda,\partial)=\partial+a\lambda+b+c$ and $h_2(\lambda,\partial)=c$, where $c\neq 0$.
According to (\ref{f11}), (\ref{f1}) becomes
\begin{align}
\label{f18}((a-1)\lambda-\mu+b)h_1(\lambda+\mu,\partial)=h_1(\mu,\lambda+\partial)(\partial+\lambda+c)+ch_1(-\lambda-\partial,\partial)
-(\partial+\lambda+\mu+c)h_1(\mu,\partial).
\end{align}
Setting $\lambda=0$ in (\ref{f18}), we can obtain $bh_1(\mu,\partial)=ch_1(-\partial,\partial)$.
Since $b\neq 0$, $h_1(\mu,\partial)=\frac{c}{b}h_1(-\partial,\partial)$.
Therefore,  when $b\neq 0$, we can assume that $h_1(\mu,\partial)=h(\partial)$, where
$h(\partial)\in \mathbb{C}[\partial]$. Taking it into (\ref{f18}), one can obtain \begin{align}
\label{f19}(\partial+a\lambda+b)h(\partial)=(\partial+\lambda+c)h(\lambda+\partial).
\end{align}
Therefore, the degree of $\partial$ in $h(\partial)$ is $0$. Then we can set
$h(\partial)=d$, where  $d\in \mathbb{C}$. By (\ref{f19}), one can have $(a-1)d=0$ and $(b-c)d=0$. Consequently, $h(\partial)=0$ or if $a=1$ and $c=b$, $h(\partial)=d$.
Thus, there are two cases: when $b\neq 0$,  $h_1(\lambda,\partial)=0$; when $b\neq 0$, $a= 1$ and $c=b$,
$h_1(\lambda,\partial)=d$, where $d\in \mathbb{C}\backslash \{0\}$.

Next, we consider the case when $b\neq 0$ in the above two cases.

{\bf Subcase b1}: when $b\neq 0$, $h_1(\lambda,\partial)=0$. Therefore, by
(\ref{f11}), $g_1(\lambda,\partial)=0$. Moreover, setting $\lambda=0$ in
(\ref{f6}), one can get $k_2(\lambda,\partial)=0$. Similarly, letting $\lambda=0$ in
 (\ref{f5}), we obtain $k_1(\lambda,\partial)=0$. Therefore, in this case, we have Case (B1). It can be verified that in this case, (\ref{f1})-(\ref{f14}) hold.

{\bf Subcase b2}: when $b\neq 0$, $a=1$ and $c=b$, $h_1(\lambda,\partial)=d$, where $d\in \mathbb{C}\backslash \{0\}$.
Then by (\ref{f11}), $g_1(\lambda,\partial)=d$. Substituting it into (\ref{f8}), we can get $k_2(\lambda,\partial)=k_2(\mu,\partial)$.
So, the degree of $\lambda$  in $k_2(\lambda,\partial)$ is 0. Therefore, assume that
$k_2(\lambda,\partial)=k_2(\partial)$, where $k_2(\partial)\in \mathbb{C}[\partial]$. Taking it into (\ref{f6}), one can get
\begin{align}
\label{f25}(b-\mu)k_2(\partial)=k_2(\lambda+\partial)(\lambda+\partial+2b)-db-(\lambda+\mu+\partial+2b)k_2(\partial).
\end{align}
By comparing the degree of $\lambda$ in (\ref{f25}), we can obtain that the degree of $\partial$ in $k_2(\partial)$ is 0. Consequently, assume  that $k_2(\partial)=e_0$, where $e_0\in \mathbb{C}$. By (\ref{f25}), we have
$e_0=-d$. Thus, $k_2(\lambda,\partial)=-d$. According to (\ref{f7}), $k_1(\lambda,\partial)=k_1(\mu,\partial)$.
Therefore, set
$k_1(\lambda,\partial)=k_1(\partial)$, where $k_1(\partial)\in \mathbb{C}[\partial]$. Taking it into (\ref{f5}), one can obtain
\begin{align}
\label{f26}(b-\mu)k_1(\partial)=k_1(\lambda+\partial)(\partial+\lambda+b)-2d^2-(\lambda+\mu+\partial+2b)k_1(\partial).
\end{align}
By (\ref{f26}), the degree of $\partial$ in $k_1(\partial)$ is 0. Therefore, set $k_1(\partial)=e_1$, where $e_1\in \mathbb{C}$. Taking it into (\ref{f26}), we can get
$e_1=-\frac{d^2}{b}$. Therefore,  $k_1(\lambda,\partial)=-\frac{d^2}{b}$. Consequently, we have Case (B2).
It can be verified that in this case, (\ref{f1})-(\ref{f14}) hold.

The proof is completed.
\end{proof}
\begin{lem}\label{lem4}
In Case (B) in Lemma \ref{lem1}, when $b=0$, we have the following results:\\
(B3) $h_1(\lambda,\partial)=g_1(\lambda,\partial)=k_1(\lambda,\partial)=k_2(\lambda,\partial)=0$, $h_2(\lambda,\partial)=c$,
and $g_2(\lambda,\partial)=\partial+a\lambda+c$;\\
(B4) When $a= 1$,
$h_1(\lambda,\partial)=g_1(\lambda,\partial)=0$, $h_2(\lambda,\partial)=c$, $k_1(\lambda,\partial)=k_1$, $k_2(\lambda,\partial)=k_2$, and
$g_2(\lambda,\partial)=\partial+\lambda+c$, where $(k_1, k_2)\in \mathbb{C}^2 \setminus\{(0,0\}$.
\end{lem}
\begin{proof}
When $b=0$,  setting $\partial=0$ in (\ref{f18}), we can get
\begin{align}
h(\mu,\lambda)(\lambda+c)=((a-1)\lambda-\mu)h_1(\lambda+\mu,0)-ch_1(-\lambda,0)+(\lambda+\mu+c)h_1(\mu,0).
\end{align}
Thus, $\lambda+c$ can divide $((a-1)\lambda-\mu)h_1(\lambda+\mu,0)-ch_1(-\lambda,0)+(\lambda+\mu+c)h_1(\mu,0)$. Set
$h_1(\mu,0)=h(\mu)$, where $h(\mu)\in\mathbb{C}[\mu]$. Therefore, we get
\begin{align}
\label{f20}h_1(\lambda,\partial)=\frac{((a-1)\partial-\lambda)h(\lambda+\partial)-ch(-\partial)+(\lambda+\partial+c)h(\lambda)}{\partial+c}.
\end{align}
Taking (\ref{f20}) into (\ref{f18}), one can obtain
\begin{align}
&&(((a-1)\lambda-\mu)((a-1)\partial-\lambda-\mu)-(\partial+c)((a-1)(\lambda+\partial)-\mu))h(\lambda+\mu+\partial)\nonumber\\
&&-c(\partial+a\lambda)h(-\partial)+(\partial+\lambda+\mu+c)((a-1)\lambda-\mu)h(\lambda+\mu)+c(\partial+\lambda)h(-\lambda-\partial)\nonumber\\
\label{f21}&&+\mu(\partial+\lambda+\mu+c)h(\mu)-c(a\partial+\lambda)h(-\lambda)+(\partial+\lambda+\mu+c)((a-1)\partial-\mu)h(\mu+\partial)=0.
\end{align}
By (\ref{f11}) and (\ref{f20}),
 \begin{align}
 g_1(\lambda,\partial)=\frac{(a\partial+\lambda)h(-\lambda)-ch(-\partial)+(c-\lambda)h(-\lambda-\partial)}{\partial+c}.
 \end{align}
 Substituting it into (\ref{f3}), we can get
 \begin{align}
 &&-(\lambda-\mu)(\lambda+\mu+\partial)h(-\lambda-\mu-\partial)+(\lambda-\mu)(a\partial+\lambda+\mu)h(-\lambda-\mu)\nonumber\\
 &&+(\lambda\partial+\lambda^2-ac\mu+a\lambda\mu)h(-\lambda-\partial)
 -(\mu\partial+\mu^2-ac\lambda+a\lambda\mu)h(-\mu-\partial)
 +ac(\mu-\lambda)h(-\partial)\nonumber\\
 &&+(\mu^2+a\mu\partial+(a^2-a)\lambda\partial+a\lambda\mu-ac\lambda)h(-\mu)\nonumber\\
\label{f22} &&-(\lambda^2+a\lambda\partial+(a^2-a)\mu\partial+a\lambda\mu-ac\mu)h(-\lambda)=0.
 \end{align}
 Set $h(\lambda)=\sum_{i=0}^mh_i\lambda^i$, where $h_m\neq 0$. If $m\geq 2$, by comparing the coefficients of $\lambda\partial\mu^m$ in (\ref{f22}), we have
 $(-1)^m h_m$ $\cdot(2C_m^2+m-1-am+a-am+(a^2-a))=0$, i.e. $(m-a)^2=1$. Therefore, $a=m+1$ or $a=m-1$. If $m\geq 3$, by comparing the coefficients of
 $\lambda^2\partial\mu^{m-1}$ and $\lambda\partial^2\mu^{m-1}$, one can get
 \begin{align}
 \label{f23}&&3C_m^3+(1-a)C_m^2+(a-2)m=0,\\
 \label{f24}&&3C_m^3+(2-a)C_m^2-m=0.
 \end{align}
 If $a=m-1$, according to (\ref{f23}) and (\ref{f24}), we can get $m=3$ and $a=2$. If $a=m+1$, by (\ref{f24}), we have $-m(\frac{m-1}{2}+1)=0$. It contradicts to $m\geq 3$. Therefore, by the above discussion, when $a=2$, we can assume  that $h(\lambda)=h_0+h_1\lambda+h_2\lambda^2+h_3\lambda^3$, where
 $h_3\neq 0$ and for the general $a$, we can assume $h(\lambda)=h_0+h_1\lambda+h_2\lambda^2$, where if $h_2\neq 0$,  $a=1$ or $a=3$. Then we discuss the following cases.

 {\bf Subcase b3}: when $b=0$ and $a=2$, $h(\lambda)=h_0+h_1\lambda+h_2\lambda^2+h_3\lambda^3$, where
 $h_3\neq 0$.
 Substituting it into (\ref{f21}) and by comparing the coefficients of $\partial\lambda^4$, we can get $h_3=0$.
 It contradicts to $h_3\neq 0$.
 Therefore, this case does not hold.

 {\bf Subcase b4}: when $b=0$ and $a=1$, $h(\lambda)=h_0+h_1\lambda+h_2\lambda^2$, where
 $h_2\neq 0$. Taking it into (\ref{f21}), we can get \begin{align}
2h_2\partial\mu\lambda(\lambda+\partial+\mu+c)+c(\partial+\lambda)(2h_2\lambda\partial-h_0)=0.
\end{align}
 Therefore, $h_2=0$. It contradicts to $h_2\neq 0$.
 Thus, this case does not hold.

{\bf Subcase b5}: when $b=0$ and $a=3$, $h(\lambda)=h_0+h_1\lambda+h_2\lambda^2$, where
 $h_2\neq 0$. Taking it into (\ref{f21}) and by comparing the coefficients of $\lambda\partial^2$, $\lambda\partial$ and $\partial$, we can obtain $h_1=3ch_2$ and $h_0=0$, i.e.
$h(\lambda)=3ch_2\lambda+h_2\lambda^2$.  Substituting it into (\ref{f20}), one can get $h_1(\lambda,\partial)=(2\partial^2+3c\partial+3c\lambda+3\partial\lambda+\lambda^2)h_2$.
Therefore, by (\ref{f11}), we have $g_1(\lambda,\partial)=(\lambda^2-3c\lambda-\lambda\partial)h_2$. Then it is easy to check that (\ref{f3}) holds. Since $c\neq 0$, by comparing the degree of $\mu$ in (\ref{f8}), we can obtain that the degree of $\lambda$ in $k_2(\lambda,\partial)$ is smaller than $3$. In addition, by comparing the degree of $\lambda$ in (\ref{f6}),
it can be directly obtained that the degree of $\partial$ in $k_2(\lambda,\partial)$ is smaller than $3$. Therefore, we can assume that
$k_2(\lambda,\partial)=(d_0+d_1\partial+d_2\partial^2)\lambda^2+(e_0+e_1\partial+e_2\partial^2)\lambda+f_0+f_1\partial+f_2\partial^2$.
Plugging it into (\ref{f8}), one can get
\begin{align}
(d_0+d_1\partial+d_2\partial^2)\mu^2+(e_0+e_1\partial+e_2\partial^2)\mu=(d_0+d_1\partial+d_2\partial^2)\lambda^2+(e_0+e_1\partial+e_2\partial^2)\lambda\nonumber\\
+(\lambda^2+\lambda\partial-\partial\mu-\mu^2)h_2.
\end{align}
Therefore, $d_1=d_2=e_0=e_2=0$ and $d_0=e_1=-h_2$, i.e. $k_2(\lambda,\partial)=-h_2\lambda^2-h_2\partial\lambda+f_0+f_1\partial+f_2\partial^2$. Putting it into (\ref{f6}) and comparing the coefficients of $\lambda^3$ and $\partial\lambda^2$, we can get
\begin{align}
-2h_2=3f_2,\\
-2h_2=7f_2.
\end{align}
Therefore, $h_2=f_2=0$. It contradicts to $h_2\neq0$. Thus, this case does not hold.

{\bf Subcase b6}: when $b=0$, $h(\lambda)=h_0+h_1\lambda$. Taking it into
(\ref{f21}) and comparing the coefficients of all terms, we get
\begin{align}
(a-1)(a-2)h_1=0\\
\label{g1}a(a-1)h_0=0,\\
\label{g2}ach_0=0.
\end{align}
Therefore, in this case, we only need to consider the following cases:
$a=1$, $h_1\neq 0$; $a=2$, $h_1\neq 0$; $h_1=0$.

When $a=1$ and $h_1\neq 0$, by (\ref{g2}), we can get $h_0=0$. Substituting $h(\lambda)=h_1\lambda$ into (\ref{f20}), one can have $h_1(\lambda,\partial)=
\frac{ch_1\partial+ch_1\lambda}{\partial+c}$. Since $\partial+c$ does not divide $ch_1\partial+ch_1\lambda$,
this case does not hold.

When $a=2$ and $h_1\neq 0$, by (\ref{g1}), we can obtain $h_0=0$. Taking $h(\lambda)=h_1\lambda$ into (\ref{f20}), we get $h_1(\lambda,\partial)=(\lambda+\partial)h_1$.  According to
(\ref{f11}), one can have $g_1(\lambda,\partial)=-h_1\lambda$. Then (\ref{f3}) holds. Since $c\neq 0$, by (\ref{f8}), it is shown that the degree of $\lambda$ in $k_2(\lambda,\partial)$ is $0$.
Moreover, by comparing the degree of $\lambda$ in (\ref{f6}), we can get that the degrees of $\lambda$ and $\partial$ are the same. Therefore, we can assume that $k_2(\lambda,\partial)=k_2$, where
$k_2\in \mathbb{C}$. Taking it into
(\ref{f6}), one can obtain
$k_2=h_1c$. Consequently, $k_2(\lambda,\partial)=h_1c$. In addition, it can be obtained from (\ref{f7}) that
\begin{align}
\label{f30} h_1^2(\partial+\lambda+\mu)(\lambda-\mu)=c(k_1(\mu,\partial)-k_1(\lambda,\partial)).
\end{align}
Setting $\mu=0$ in (\ref{f30}), one can get $k_1(\lambda,\partial)=k_1(0,\partial)-\frac{h_1^2}{c}\lambda(\partial+\lambda)$.
Therefore, we can assume that $k_1(\lambda,\partial)=p(\partial)-\frac{h_1^2}{c}\lambda(\partial+\lambda)$, where $p(\partial)\in \mathbb{C}[\partial]$. Taking it into (\ref{f5}), we obtain
\begin{align}
\label{f31}(\partial+3\lambda+c)p(\partial)-(\partial+\lambda+c)p(\lambda+\partial)=\frac{h_1^2}{c}(\lambda^2\partial+\lambda^3)-ch_1^2\lambda
+h_1^2\lambda\partial.
\end{align}
By comparing the coefficients of $\lambda$, one can set $p(\partial)=e_0+e_1\partial+e_2\partial^2$, where $e_0$, $e_1$ and $e_2\in \mathbb{C}$. Plugging it into (\ref{f31}) and according to the coefficients of
$\lambda^3$ and $\lambda^2\partial$, we get that $e_2=-\frac{h_1^2}{c}$ and $3e_2=-\frac{h_1^2}{c}$. Since
$h_1$ and $c$ are not equal to 0, we get a contradiction. Therefore, this case does not hold.

When $h_1=0$, substituting $h(\lambda)=h_0$ into (\ref{f20}), we have $h_1(\lambda,\partial)=
h_0\frac{a\partial}{\partial+c}$. Since $c\neq 0$,  $h_1(\lambda,\partial)=0$.
Therefore, $g_1(\lambda,\partial)=0$. By comparing the degree of $\lambda$ in (\ref{f8}), we can know that the degree of $\lambda$ in $k_2(\lambda,\partial)$ is $0$. Assume that $k_2(\lambda,\partial)=k_2(\partial)$, where $k_2(\partial)\in\mathbb{C}[\partial]$. Taking it into
(\ref{f6}), one can get
\begin{align}
((2a-1)\lambda+\partial+c)k_2(\partial)=(a\lambda+\partial+c)k_2(\lambda+\partial).
\end{align}
Consequently, when $a=1$, $k_2(\lambda,\partial)=k_2$, where $k_2\in \mathbb{C}$, or when $a\neq 1$, $k_2(\lambda,\partial)=0$. By comparing the degree of $\lambda$ in (\ref{f7}), we get that the degree of $\lambda$ in $k_1(\lambda,\partial)$ is $0$.
 Therefore, set $k_1(\lambda,\partial)=k_1(\partial)$, where $k_1(\partial)\in \mathbb{C}[\partial]$. Plugging it into
(\ref{f5}), we get
\begin{align}
((2a-1)\lambda+\partial+c)k_1(\partial)=(\lambda+\partial+c)k_1(\lambda+\partial).
\end{align}
Therefore, when $a=1$, $k_1(\lambda,\partial)=k_1$, where $k_1\in \mathbb{C}$, or when $a\neq 1$, $k_1(\lambda,\partial)=0$. Therefore, we get Case (B3) and Case (B4). It is easy to check that in both two cases, (\ref{f1})-(\ref{f14}) hold.

The proof is finished.
\end{proof}

\begin{lem}\label{lem5}
In Case (C) in Lemma \ref{lem1}, when $b\neq 0$, we have the following result:\\
(C1) $g_2(\lambda,\partial)=\partial+(a-1)\lambda+b+c$, $h_2(\lambda,\partial)=\partial+\lambda+c$, $h_1(\lambda,\partial)=g_1(\lambda,\partial)=k_1(\lambda,\partial)=k_2(\lambda,\partial)=0$.
\end{lem}
\begin{proof}
In Case (C) in Lemma \ref{lem1}, $g_2(\lambda,\partial)=\partial+(a-1)\lambda+b+c$, $h_2(\lambda,\partial)=\partial+\lambda+c$. By (\ref{f11}),  (\ref{f1}) becomes
\begin{align}
((a-1)\lambda-\mu+b)h_1(\lambda+\mu,\partial)=h_1(\mu,\lambda+\partial)(\partial+\lambda+c)+(\mu+\lambda+\partial+c)h_1(-\lambda-\partial,\partial)\nonumber\\
\label{g7}-(\partial+\lambda+\mu+c)h_1(\mu,\partial).
\end{align}
Setting $\lambda=0$ in (\ref{g7}), one can obtain
\begin{align}
\label{h1}bh_1(\mu,\partial)=(\mu+\partial+c)h_1(-\partial,\partial).
\end{align}

Since $b\neq 0$, according to (\ref{h1}), we can assume that $h_1(\mu,\partial)=(\mu+\partial+c)p(\partial)$, where $p(\partial)\in \mathbb{C}[\partial]$. Plugging it into (\ref{g7}), we can get
\begin{align}
\label{h2}(\partial+a\lambda+b)p(\partial)=(\partial+\lambda+c)p(\lambda+\partial).
\end{align}
By (\ref{h2}), when $a=1$ and $c=b$, $p(\partial)=d_0$, where $d_0\in \mathbb{C}$; when $a\neq 1$ or $c\neq b$, $p(\partial)=0$.
Therefore, when $a=1$ and $c=b$, $h_1(\mu,\partial)=(\mu+\partial+c)d_0$; when $a\neq 1$ or $c\neq b$, $h_1(\mu,\partial)=0$. Then we can discuss the case when $b\neq 0$ in the following cases.

{\bf Subcase c1}: when $b\neq 0$ and $a\neq 1$ or $c\neq b$, $h_1(\mu,\partial)=0$. Hence, $g_1(\lambda,\partial)=0$. By comparing the degree of $\lambda$
in (\ref{f7}), one  can get that the degree of $\lambda$ in $k_1(\lambda,\partial)$
is $0$. Consequently, we assume that $k_1(\lambda,\partial)=k_1(\partial)$, where $k_1(\partial)\in \mathbb{C}[\partial]$. Taking it into (\ref{f5}), we have
\begin{align}
(2(a-1)\lambda+\partial+2b+c)k_1(\partial)=k_1(\lambda+\partial)(\partial+\lambda+c).
\end{align}
Therefore, $k_1(\lambda,\partial)=k_1(\partial)=0$. By comparing the degree of $\mu$ in (\ref{f10}), we can obtain
that the degree of $\partial$ in $k_2(\lambda,\partial)$ is $0$. In addition, it can be obtained that the degree of
$\lambda$ in $k_2(\lambda,\partial)$ is $0$, by comparing the degree of $\mu$ in (\ref{f8}). So, assume that
$k_2(\lambda,\partial)=k_2$, where $k_2\in \mathbb{C}$. Putting it into (\ref{f6}), one can obtain that
$k_2((a-1)\lambda+b)=0$, i.e. $k_2=0$. Consequently, $k_2(\lambda,\partial)=0$. Therefore, in this case, $g_2(\lambda,\partial)=\partial+(a-1)\lambda+b+c$, $h_2(\lambda,\partial)=\partial+\lambda+c$, $h_1(\lambda,\partial)=g_1(\lambda,\partial)=k_1(\lambda,\partial)=k_2(\lambda,\partial)=0$.

{\bf Subcase c2}: when $b\neq 0$, $a=1$ and $c=b$, $h_1(\mu,\partial)=(\mu+\partial+c)d_0$. By
(\ref{f11}), one can get $g_1(\lambda,\partial)=(c-\lambda)d_0$.
According to (\ref{f8}), we have
\begin{align}
(\lambda-\mu)d_0+k_2(\lambda,\partial)=k_2(\mu,\partial).
\end{align}
Therefore, $k_2(\lambda,\partial)=-d_0\lambda+p(\partial)$, where $p(\partial)\in \mathbb{C}[\partial]$. Taking it into
(\ref{f6}), we obtain
\begin{align}
(c-\mu)(-d_0(\lambda+\mu)+p(\partial))=(\partial+2c)(-d_0\mu+p(\lambda+\partial))-d_0(c-\lambda)(\mu+\partial+c)\nonumber\\
\label{g8}-(\mu+\partial+2c)(-d_0\mu+p(\partial)).
\end{align}
By comparing the degree of $\lambda$ in (\ref{g8}), one can get that the degree of $\partial$ in $p(\partial)$ is smaller than
$2$. Therefore, assume that $p(\partial)=p_0+p_1\partial$, where $p_0$, $p_1\in \mathbb{C}$.
Plugging it into (\ref{g8}), we obtain
\begin{align}
cp_0+d_0c^2+c(p_1+d_0)\partial=(p_1+d_0)\partial\lambda+2c(p_1+d_0)\lambda.
\end{align}
Therefore, $p_0=-cd_0$ and $p_1=-d_0$, i.e. $k_2(\lambda,\partial)=-d_0(\lambda+\partial+c)$.
Taking it into (\ref{f5}), we have
\begin{align}
(c-\mu)k_1(\lambda+\mu,\partial)=k_1(\mu,\lambda+\partial)(\partial+\lambda+c)-d_0^2(c-\lambda)(2\mu+\lambda+2\partial+2c)\nonumber\\
\label{h3}-k_1(\mu,\partial)(\mu+\partial+2c).
\end{align}
Setting $\lambda=0$ in (\ref{h3}), one can get $k_1(\mu,\partial)=-d_0^2(\mu+\partial+c)$.
Substituting $k_1(\lambda, \partial)$ and $k_2(\lambda,\partial)$ into (\ref{f13}) and (\ref{f14}) respectively, we have $d_0=0$.
 Therefore,  $k_1(\lambda,\partial)=k_2(\lambda,\partial)=0$. As a result, $g_2(\lambda,\partial)=\partial+2b$, $h_2(\lambda,\partial)=\partial+\lambda+b$, $h_1(\lambda,\partial)=g_1(\lambda,\partial)=k_1(\lambda,\partial)=k_2(\lambda,\partial)=0$.

According to the results in Subcase c1 and Subcase c2, we obtain this lemma.
\end{proof}

Finally, we will consider Case (C) when $b=0$.
\begin{lem}\label{lem6}
In Case (C) in Lemma \ref{lem1}, when $b=0$, we have the following results:\\
(C2) $h_1(\lambda,\partial)=g_1(\lambda,\partial)=k_1(\lambda,\partial)=k_2(\lambda,\partial)=0$, $g_2(\lambda,\partial)
 =\partial+(a-1)\lambda+c$ and $h_2(\lambda,\partial)=\partial+\lambda+c$;\\
(C3) When $a=1$, $g_2(\lambda,\partial)=\partial+c$,
$h_2(\lambda,\partial)=\partial+\lambda+c$, $h_1(\lambda,\partial)=g_1(\lambda,\partial)=k_1(\lambda,\partial)=0$, $k_2(\lambda,\partial)=k_2$, where $k_2\in \mathbb{C}\backslash \{0\}$.
\end{lem}

\begin{proof}
When $b=0$, taking $\partial=0$ in (\ref{g7}), we have
\begin{align}
(\lambda+c)h_1(\mu,\lambda)=((a-1)\lambda-\mu)h_1(\lambda+\mu,0)-(\mu+\lambda+c)h_1(-\lambda,0)+(\mu+\lambda+c)h_1(\mu,0).
\end{align}
Therefore, $\lambda+c$ can divide $((a-1)\lambda-\mu)h_1(\lambda+\mu,0)-(\mu+\lambda+c)h_1(-\lambda,0)+(\lambda+\mu+c)h_1(\mu,0)$. Set
$h_1(\mu,0)=h(\mu)$, where $h(\mu)\in\mathbb{C}[\mu]$. Thus, we get
\begin{align}
\label{g9}h_1(\lambda,\partial)=\frac{((a-1)\partial-\lambda)h(\lambda+\partial)-(\partial+\lambda+c)h(-\partial)+(\lambda+\partial+c)h(\lambda)}{\partial+c}.
\end{align}
According to (\ref{f11}) and (\ref{g9}),
\begin{align}
\label{g10}g_1(\lambda,\partial)=\frac{(a\partial+\lambda)h(-\lambda)-(c-\lambda)h(-\partial)+(c-\lambda)h(-\lambda-\partial)}{\partial+c}.
\end{align}
Taking it into (\ref{f3}) and after some computations, we can obtain
\begin{align}
(\mu^2-\lambda^2+\partial\mu-\partial\lambda)h(-\lambda-\mu-\partial)+(\lambda^2+a\lambda\partial-a\mu\partial-\mu^2)h(-\lambda-\mu)\nonumber\\
+(\lambda^2-\partial\mu-ac\mu+(a-1)\mu\lambda+\partial\lambda)h(-\lambda-\partial)
+(ac\lambda-(a-1)\lambda\mu-\mu^2-\partial\mu+\partial\lambda)h(-\mu-\partial)\nonumber\\
+(ac\mu-\lambda\partial-ac\lambda+\partial\mu)h(-\partial)+(a(a-2)\lambda\partial+(a-1)\lambda\mu+a\mu\partial+\mu^2-ac\lambda)h(-\mu)\nonumber\\
\label{g11}+(ac\mu-a(a-2)\partial\mu-a\partial\lambda-(a-1)\lambda\mu-\lambda^2)h(-\lambda)=0.
\end{align}
Obviously, $h(\lambda)=0$ is a solution of (\ref{g11}).  Assume that $h(\lambda)=\sum_{i=0}^mh_i\lambda^i$, where $h_m\neq 0$. If $m\geq 2$, by comparing the coefficients of $\lambda\partial\mu^m$ in (\ref{g11}), one can have
\begin{align}
(m-a)(m-a+1)=0.
\end{align}
If $m\geq 3$, comparing the coefficients of
$\lambda^2\partial\mu^{m-1}$ and $\lambda\partial^2\mu^{m-1}$ in (\ref{g11}), we get
\begin{align}
m(m-3)(m-a+1)=0,\\
m(m-1)(m-a+1)=0,
\end{align}
Similarly, if $m\geq 4$, by comparing the coefficients of $\partial^2\lambda^2\mu^{m-2}$ in (\ref{g11}), we obtain
\begin{align}
m(m-1)(m+1)(m-4)=0.
\end{align}
Therefore, we only need to discuss the following cases:  $m=4$, $a=5$ and $h_4 \neq 0$; $m=3$, $a=4$ and $h_3 \neq 0$;
$m=2$, $a=2$ and $h_2 \neq 0$; $m=2$, $a=3$ and $h_2 \neq 0$; $m\leq 1$.

Setting $\mu=-\lambda=-\partial$ in (\ref{g7}), we get
\begin{align}
\label{g16}a\partial h_1(0,\partial)=(2\partial+c)h_1(-\partial,2\partial)+(\partial+c)h_1(-2\partial,\partial)-(\partial+c)h_1(-\partial,\partial).
\end{align}
Then by (\ref{h1}), we get $h_1(-\partial,\partial)=0$. Substituting (\ref{g9}) into (\ref{g16}), we can obtain
\begin{align}
\label{g12}((a^2-3a+1)\partial-(2a-1)c)h(\partial)+a(\partial+c)h(0)=(2a+3)(\partial+c)h(-\partial)-2(\partial+c)h(-2\partial).
\end{align}

{\bf Subcase c3}: $m=4$, $a=5$ and $h_4 \neq 0$. According to (\ref{g12}), one can get
\begin{align}
\label{g13}(11\partial-9c)h(\partial)+5(\partial+c)h(0)=13(\partial+c)h(-\partial)-2(\partial+c)h(-2\partial).
\end{align}
Taking $h(\lambda)=\sum_{i=0}^4h_i\lambda^i$ into (\ref{g13}), and by comparing the coefficients of $\partial^5$ in (\ref{g13}), we can have $h_4=0$. Therefore, we get a contradiction. As a result, this case does not hold.

{\bf Subcase c4}: $m=3$, $a=4$ and $h_3 \neq 0$. By (\ref{g12}), we obtain
\begin{align}
\label{g14}(5\partial-7c)h(\partial)+4(\partial+c)h(0)=11(\partial+c)h(-\partial)-2(\partial+c)h(-2\partial).
\end{align}
Taking $h(\lambda)=\sum_{i=0}^3h_i\lambda^i$ into (\ref{g14}), and by comparing the coefficients of $\partial^4$, one can get $h_3=0$. It contradicts to our assumption. Therefore, this case does not hold.

{\bf Subcase c5}: $m=2$, $a=2$ and $h_2 \neq 0$. By (\ref{g12}), we have
\begin{align}
\label{g15}-(\partial+3c)h(\partial)+2(\partial+c)h(0)=7(\partial+c)h(-\partial)-2(\partial+c)h(-2\partial).
\end{align}
Plugging $h(\lambda)=\sum_{i=0}^2h_i\lambda^i$ into (\ref{g15}), and by comparing the coefficients of constant term, $\partial$, and $\partial^2$, we can obtain that $h_0=0$ and $h_1=ch_2$, i.e. $h(\lambda)=h_2(c\lambda+\lambda^2)$. Substituting it into (\ref{g9}), one can obtain
\begin{align}
h_1(\lambda,\partial)=ch_2\frac{(\partial+\lambda)(\partial+\lambda+c)}{\partial+c}.
\end{align}
 Therefore, $h_1(\lambda,\partial)=0$ and $c=0$. By (\ref{f11}), $g_1(\lambda,\partial)=0$.
 By comparing the coefficients of $\lambda$ in (\ref{f7}), we can assume that $k_1(\lambda,\partial)=k_1(\partial)$,
  Then, according to (\ref{f5}), one can obtain
\begin{align}
\label{g17}(2\lambda+\partial)k_1(\partial)=(\lambda+\partial)k_1(\lambda+\partial).
\end{align}
 Consequently, $k_1(\lambda,\partial)=k_1(\partial)=0$. Because (\ref{f6}), (\ref{f8}) and (\ref{f5}), (\ref{f7}) are the same, we can get $k_2(\lambda,\partial)=0$ in a similar way. Therefore, we obtain Case (C2) with $a=2$.

{\bf Subcase c6}: $m=2$, $a=3$ and $h_2 \neq 0$. By (\ref{g12}), one can get
\begin{align}
\label{g18}(\partial-5c)h(\partial)+3(\partial+c)h(0)=9(\partial+c)h(-\partial)-2(\partial+c)h(-2\partial).
\end{align}
Taking $h(\lambda)=\sum_{i=0}^2h_i\lambda^i$ into (\ref{g18}) and comparing the coefficients of constant term, $\partial$ and $\partial^2$, we can get that $h_0=0$ and $h_1=ch_2$ , i.e. $h(\lambda)=h_2(c\lambda+\lambda^2)$. Substituting it into (\ref{g9}), we get $h_1(\lambda,\partial)=h_2(\lambda+\partial)(\lambda+\partial+c)$. By (\ref{f11}), we have
$g_1(\lambda,\partial)=h_2\lambda(\lambda-c)$. Then it is easy to check that (\ref{f1}) and (\ref{f3}) hold. By (\ref{f8}),
\begin{align}
h_2(\lambda+\mu+\partial)(\lambda-\mu)+k_2(\lambda,\partial)=k_2(\mu,\partial).
\end{align}
Therefore, $k_2(\lambda,\partial)=-h_2(\lambda^2+\lambda\partial)+p(\partial)$ , where $p(\partial)\in \mathbb{C}[\partial]$.
 Taking it into (\ref{f6}), we obtain
\begin{align}
(2\lambda-\mu)(-h_2(\lambda+\mu)(\lambda+\mu+\partial)+p(\partial))=(2\lambda+\partial+c)(-h_2\mu(\mu+\partial+\lambda)+p(\lambda+\partial))\nonumber\\
\label{g19}-h_2\lambda(\lambda-c)(\partial+\mu+c)-(2\lambda+\mu+\partial+c)(-h_2\mu(\mu+\partial)+p(\partial)).
\end{align}
By comparing the degree of  $\lambda$  in (\ref{g19}), it is easy to see that the degree of $p(\partial)$ is smaller than $3$. Therefore, assume  that $p(\partial)=p_0+p_1\partial+p_2\partial^2$. Substituting it into (\ref{g19}), and by comparing the coefficients of $\lambda^3$, we get $p_2=-h_2$. By comparing the coefficients of $\lambda^2\partial$, we can get $h_2=0$. So, we get a contradiction. Therefore, this case does not hold.

{\bf Subcase c7}: $m\leq 1$. Putting $h(\lambda)=h_0+h_1\lambda$ into (\ref{g9}), we get
\begin{align}
\label{g20}h_1(\lambda,\partial)=\frac{(a-1)\partial(h_1\partial+h_0)+\lambda((a-1)h_1\partial-h_0)}{\partial+c}+h_1\lambda+h_1\partial.
\end{align}
Then we can give a discussion about whether $h_1$ is equal to $0$.

When $h_1 \neq 0$, according to that $\partial+c$ can divide $(a-1)\partial(h_1\partial+h_0)+\lambda((a-1)h_1\partial-h_0)$, we can discuss the following cases:  $a=1$, $h_0=0$; $a=0$, $h_0=ch_1$ and $c=0$, $h_0=0$.

When $c=0$ and $h_0=0$, by (\ref{g20}), we have $h_1(\lambda,\partial)=ah_1(\lambda+\partial)$. Taking it into (\ref{g7}),
 we can obtain $a(a-1)h_1\lambda=0$. Consequently, $a=0$ or $a=1$. Therefore, this case when $a=0$ can be merged into  the case when $a=0$ and $h_0=ch_1$ , and this case when $a=1$ can be merged into the case when $a=1$ and $h_0=0$.

When $a=1$ and $h_0=0$, by (\ref{g20}),  $h_1(\lambda,\partial)=h_1(\lambda+\partial)$. According to (\ref{f11}), we get $g_1(\lambda,\partial)=-h_1\lambda$. Then it is easy to check that
(\ref{g7}) and (\ref{f3}) hold. Setting $\partial=0$ in (\ref{f8}), we can obtain
\begin{align}
\label{g21}h_1(\mu+\lambda)(\lambda-\mu)+(\mu+\lambda+c)k_2(\lambda,0)=(\mu+\lambda+c)k_2(\mu,0).
\end{align}
Therefore, according to (\ref{g21}), we can assume that $k_2(\lambda,0)=k_2(\lambda)=k_0+k_1\lambda$, where $k_0$, $k_1\in \mathbb{C}$. Taking it into (\ref{g21}), we obtain $h_1(\lambda+\mu)+k_1(\lambda+\mu+c)=0$. Thus,
$k_1=-h_1$ and $c=0$. Setting $\partial=0$ in (\ref{f13}), and substituting $k_2(\lambda)=k_0-h_1\lambda$ into it, one can get $h_1=0$. We get a contradiction. Therefore, this case does not hold.

When $a=0$ and $h_0=ch_1$, by (\ref{g20}), we get $h_1(\lambda,\partial)=0$. When $h_1=0$, according to (\ref{g20}), we also get $h_1(\lambda,\partial)=0$.
Therefore, we can discuss the two cases together.

When $b=0$, $h_1(\lambda,\partial)=0$. By (\ref{f11}), $g_1(\lambda,\partial)=0$. Then (\ref{g7}) and (\ref{f3}) hold.
 By comparing the degree of  $\lambda$ in (\ref{f7}), we can get that the degree of $\lambda$ in $k_1(\lambda,\partial)$ is $0$.
 Therefore, set $k_1(\lambda,\partial)=k_1(\partial)$. Substituting it into (\ref{f5}), we have
\begin{align}
(2(a-1)\lambda+\partial+c)k_1(\partial)=(\partial+\lambda+c)k_1(\lambda+\partial).
\end{align}
Therefore, we get that $k_1(\lambda,\partial)=0$ or if $a=\frac{3}{2}$, $k_1(\lambda,\partial)=k_1$, where $k_1\in \mathbb{C}$. By comparing the degree of $\lambda$ in (\ref{f8}), we can get that the degree of $\lambda$ in $k_2(\lambda,\partial)$ is $0$. Therefore, set $k_2(\lambda,\partial)=k_2(\partial)$, where $ k_2(\partial)\in \mathbb{C}[\partial]$. Taking it into (\ref{f6}), we get
\begin{align}
(2(a-1)\lambda+\partial+c)k_2(\partial)=((a-1)\lambda+\partial+c)k_2(\lambda+\partial).
\end{align}
Therefore, $k_2(\lambda,\partial)=0$ or when $a=1$, $k_2(\lambda,\partial)=k_2$, where $k_2\in \mathbb{C}$. In addition, by (\ref{f10}), $k_1(\lambda,\partial)=0$. Therefore, in this case, we have Case (C2) and Case (C3). It is easy to check that in the two cases, (\ref{f1})-(\ref{f14}) hold.

The proof is finished.
\end{proof}

\begin{thm}\label{t1}
When $a\neq 1$, all compatible left-symmetric conformal algebraic structures over $\mathcal{W}(a,b)$ such that $\mathbb{C}[\partial]L$ is a left-symmetric conformal subalgebra  are as follows\\
(1) \begin{align}
L_\lambda L=(\partial+\lambda+c)L,~~~~~L_\lambda W=(\partial+a\lambda+b)W,\\
W_\lambda L=0,~~~~~W_\lambda W=0,
\end{align}
where $c\in \mathbb{C}$;\\
(2) \begin{align}
L_\lambda L=(\partial+\lambda+c)L,~~~~~L_\lambda W=(\partial+a\lambda+b+c)W,\\
W_\lambda L=cW,~~~~~W_\lambda W=0,
\end{align}
where $c\in \mathbb{C}\backslash\{0\}$;\\
(3) \begin{align}
L_\lambda L=(\partial+\lambda+c)L,~~~~~L_\lambda W=(\partial+(a-1)\lambda+b+c)W,\\
W_\lambda L=(\partial+\lambda+c)W,~~~~~W_\lambda W=0,
\end{align}
where $c\in \mathbb{C}$.

When $b\neq 0$, all compatible left-symmetric conformal algebraic structures over $\mathcal{W}(1,b)$ such that $\mathbb{C}[\partial]L$ is a left-symmetric conformal subalgebra  are as follows: Case (1) when $a=1$,
Case (2) when $a=1$, Case (3) when $a=1$, and \\
(4) \begin{align}
L_\lambda L=(\partial+\lambda+2b)L,~~~~~L_\lambda W=(\partial+\lambda+b)W,\\
W_\lambda L=0,~~~~~W_\lambda W=k_1L,
\end{align}
where $k_1\neq 0$;\\
(5) \begin{align}
L_\lambda L=(\partial+\lambda+b)L,~~~~~L_\lambda W=dL+(\partial+\lambda+2b)W,\\
W_\lambda L=dL+bW,~~~~~W_\lambda W=-\frac{d^2}{b}L-dW,
\end{align}
where $d\neq 0$.

All compatible left-symmetric conformal algebraic structures over $\mathcal{W}(1,0)$ such that $\mathbb{C}[\partial]L$ is a left-symmetric conformal subalgebra  are as follows: Case (1) when $a=1$  and $b=0$, Case (2) when $a=1$ and $b=0$, Case (3) when $a=1$ and $b=0$, Case (4) when $b=0$, and\\
(6)\begin{align}
L_\lambda L=(\partial+\lambda+c)L,~~~~~L_\lambda W=(\partial+\lambda)W,\\
W_\lambda L=0,~~~~~W_\lambda W=k_2W,
\end{align}
where $k_2\in \mathbb{C}\backslash \{0\}$;\\
(7)\begin{align}
L_\lambda L=(\partial+\lambda)L,~~~~~L_\lambda W=(\partial+\lambda)W,\\
W_\lambda L=0,~~~~~W_\lambda W=k_1L+k_2W,
\end{align}
where $k_1$, $k_2\in \mathbb{C}\backslash \{0\}$;\\
(8)\begin{align}
L_\lambda L=(\partial+\lambda+c)L,~~~~~L_\lambda W=h_1L+(\partial+\lambda)W,\\
W_\lambda L=h_1L,~~~~~W_\lambda W=\frac{h_1(h_1-k_2)}{c}L+k_2W,
\end{align}
where $c$, $h_1\in \mathbb{C}\backslash \{0\}$, $k_2\in \mathbb{C}$;\\
(9)\begin{align}
L_\lambda L=(\partial+\lambda)L,~~~~~L_\lambda W=h_1L+(\partial+\lambda)W,\\
W_\lambda L=h_1L,~~~~~W_\lambda W=k_1L+h_1W,
\end{align}
where  $h_1$, $k_1\in \mathbb{C}\backslash \{0\}$;\\
(10)\begin{align}
L_\lambda L=(\partial+\lambda+c)L,~~~~~L_\lambda W=(\partial+\lambda+c)W,\\
W_\lambda L=cW,~~~~~W_\lambda W=k_1L+k_2W,
\end{align}
where $c\in \mathbb{C}\backslash \{0\}$, $(k_1,k_2)\in \mathbb{C}^2\backslash(0,0)$;\\
(11) \begin{align}
L_\lambda L=(\partial+\lambda+c)L,~~~~~L_\lambda W=(\partial+c)W,\\
W_\lambda L=(\partial+\lambda+c)W,~~~~~W_\lambda W=k_2W,
\end{align}
where $c\in \mathbb{C}$, $k_2\in \mathbb{C}\backslash\{0\}$.
\end{thm}
\begin{proof}
This theorem can be directly obtained by Lemma \ref{lem2}, Lemma \ref{lem3}, Lemma \ref{lem4}, Lemma \ref{lem5} and Lemma \ref{lem6}.
\end{proof}

\begin{cor}
Coeff$(\mathcal{W}(a,b))$ has the following compatible left-symmetric algebraic structures:\\
(1)\begin{align}
L_m\circ L_n=cL_{m+n+1}-(n+1)L_{m+n},~~~W_m\circ L_n=0,\\
~L_m\circ W_n=((a-1)(m+1)-(n+1))W_{m+n}+bW_{m+n+1},~~~~W_m\circ W_n=0,
\end{align}
where $c\in \mathbb{C}$;\\
(2)\begin{align}
L_m\circ L_n=cL_{m+n+1}-(n+1)L_{m+n},~~~W_m\circ L_n=cW_{m+n+1},\\
~L_m\circ W_n=((a-1)(m+1)-(n+1))W_{m+n}+(b+c)W_{m+n+1},~~~~W_m\circ W_n=0,
\end{align}
where $c\in \mathbb{C}\backslash\{0\}$;\\
(3)\begin{align}
L_m\circ L_n=cL_{m+n+1}-(n+1)L_{m+n},~~~~W_m\circ L_n=-(n+1)W_{m+n}+cW_{m+n+1},\\
L_m\circ W_n=((a-2)(m+1)-(n+1))W_{m+n}+(b+c)W_{m+n+1},~~~~W_m\circ W_n=0,
\end{align}
where $c\in \mathbb{C}$.

Coeff$(\mathcal{W}(1,b))$ also has the following compatible left-symmetric algebraic structures\\
(4)\begin{align}
L_m\circ L_n=2bL_{m+n+1}-(n+1)L_{m+n},~~~~L_m\circ W_n=-(n+1)W_{m+n}+bW_{m+n+1},\\
W_m\circ L_n=0,~~~~W_m\circ W_n=k_1L_{m+n+1},
\end{align}
where $k_1\neq 0$.

In particular, when $b\neq 0$, there also exist the following compatible left-symmetric algebraic structures on Coeff$(\mathcal{W}(1,b))$\\
(5)\begin{align}
L_m\circ L_n=bL_{m+n+1}-(n+1)L_{m+n},~~~~L_m\circ W_n=dL_{m+n+1}-(n+1)W_{m+n}+2bW_{m+n+1},\\
W_m\circ L_n=dL_{m+n+1}+bW_{m+n+1},~~~~W_m\circ W_n=-\frac{d^2}{b}L_{m+n+1}-dW_{m+n+1},
\end{align}
where $d\neq 0$.

Moreover, there also exist the following compatible left-symmetric algebraic structures on Coeff$(\mathcal{W}(1,0))$:\\
(6)\begin{align}
L_m\circ L_n=cL_{m+n+1}-(n+1)L_{m+n},~~~~L_m\circ W_n=-(n+1)W_{m+n},\\
W_m\circ L_n=0,~~~~W_m\circ W_n=k_2W_{m+n+1},
\end{align}
where $k_2\in \mathbb{C}\backslash \{0\}$;\\
(7)\begin{align}
L_m\circ L_n=-(n+1)L_{m+n},~~~~L_m\circ W_n=-(n+1)W_{m+n},\\
W_m\circ L_n=0,~~~~W_m\circ W_n=k_1L_{m+n+1}+k_2W_{m+n+1},
\end{align}
where $k_1$, $k_2\in \mathbb{C}\backslash \{0\}$;\\
(8)\begin{align}
L_m\circ L_n=cL_{m+n+1}-(n+1)L_{m+n},~~~~L_m\circ W_n=h_1L_{m+n+1}-(n+1)W_{m+n},\\
W_m\circ L_n=h_1L_{m+n+1},~~~~W_m\circ W_n=\frac{h_1(h_1-k_2)}{c}L_{m+n+1}+k_2W_{m+n+1},
\end{align}
where $c$, $h_1\in \mathbb{C}\backslash \{0\}$, $k_2\in \mathbb{C}$;\\
(9)\begin{align}
L_m\circ L_n=cL_{m+n+1}-(n+1)L_{m+n},~~~~L_m\circ W_n=h_1L_{m+n+1}-(n+1)W_{m+n},\\
W_m\circ L_n=h_1L_{m+n+1},~~~~W_m\circ W_n=k_1L_{m+n+1}+h_1W_{m+n+1},
\end{align}
where $h_1$, $k_1\in \mathbb{C}\backslash\{0\}$;\\
(10)\begin{align}
L_m\circ L_n=cL_{m+n+1}-(n+1)L_{m+n},,~~~~L_m\circ W_n=cW_{m+n+1}-(n+1)W_{m+n},\\
W_m\circ L_n=cW_{m+n+1},~~~~W_m\circ W_n=k_1L_{m+n+1}+k_2W_{m+n+1},
\end{align}
where $c\in \mathbb{C}\backslash\{0\}$, $(k_1,k_2)\in \mathbb{C}^2\backslash (0,0)$;\\
(11) \begin{align}
L_m\circ L_n=cL_{m+n+1}-(n+1)L_{m+n},~~~~L_m\circ W_n=-(m+n+2)W_{m+n}+cW_{m+n+1},\\
W_m\circ L_n=-(n+1)W_{m+n}+cW_{m+n+1},~~~~W_m\circ W_n=k_2W_{m+n+1},
\end{align}
where $c\in \mathbb{C}$, $k_2\in \mathbb{C}\backslash\{0\}$.
\end{cor}
\begin{proof}
This corollary can be directly obtained by Theorem \ref{t1} and Proposition \ref{pp1}.
\end{proof}

{\bf Acknowledgments}
{This work was supported by the Zhejiang Provincial Natural Science Foundation of China (No. LQ16A010011), the National Natural Science Foundation of China (No. 11501515, Tianyuan fund for mathematics No. 11626216),  the Scientific Research Foundation of Zhejiang Agriculture and Forestry University (No. 2013FR081) and the research and training program for students at Zhejiang Agriculture and Forestry University.}

\end{document}